\documentclass[10pt,draft,twoside]{amsart}
\usepackage{amssymb}
\usepackage{latexsym}
\usepackage{amsfonts}
\usepackage{amsmath}

\DeclareMathOperator{\Aut}{Aut}
\DeclareMathOperator{\rep}{rep}
\DeclareMathOperator{\Rep}{Rep}

\DeclareMathOperator{\PSL}{PSL}
\DeclareMathOperator{\PGaL}{P\Gamma L}
\DeclareMathOperator{\Form}{Form}

\DeclareMathOperator{\soc}{soc}
\DeclareMathOperator{\Inn}{Inn}

\DeclareMathOperator{\Prod}{Prod}
\DeclareMathOperator{\Diag}{Diag}

\DeclareMathOperator{\All}{All}

\DeclareMathOperator{\case1}{{\sf{Case\,1}}}
\DeclareMathOperator{\taucase}{{\sf{Case\,2}}}

\DeclareMathOperator{\Sym}{Sym}

\newcommand{\bs}{\boldsymbol}

\newtheorem{theorem}{Theorem}[section]
\newtheorem{proposition}[theorem]{Proposition}
\newtheorem{lemma}[theorem]{Lemma}
\newtheorem{corollary}[theorem]{Corollary}
\newtheorem{notation}[theorem]{Notation}

\theoremstyle{definition}

\newtheorem{remark}[theorem]{Remark}
\newtheorem{definition}[theorem]{Definition}
\newtheorem{example}[theorem]{Example}

\renewcommand{\wr}{\,{\sf wr}\,}

\renewcommand{\leq}{\leqslant}
\renewcommand{\geq}{\geqslant}

\numberwithin{equation}{section}

\begin{document}
\title[Characterisation of a family of neighbour transitive codes] 
      {Characterisation of a family of neighbour transitive codes} 
\author{Neil I. Gillespie and Cheryl E. Praeger}
\address{[Gillespie] Heilbronn Institute for Mathematical Research, School of Mathematics,
Howard House, The University of Bristol, Bristol, BS8 1SN, United Kingdom.}
\email{neil.gillespie@bristol.ac.uk}
\address{[Praeger] Centre for the Mathematics of Symmetry and Computation
School of Mathematics and Statistics
The University of Western Australia
35 Stirling Highway, Crawley, Western Australia, 6009.
Also affiliated with King Abdulaziz University, Jeddah, Saudi Arabia.}
\email{cheryl.praeger@uwa.edu.au}

\begin{abstract}  
We consider codes of length $m$ over an alphabet of size $q$ as subsets of the vertex set of the Hamming graph $\Gamma=H(m,q)$.  
A code for which there exists an automorphism group $X\leq\Aut(\Gamma)$ that acts transitively on the code and 
on its set of neighbours is said to be \emph{neighbour transitive}, and were introduced 
by the authors as a group theoretic analogue to the assumption that single errors are 
equally likely over a noisy channel. Examples of neighbour transitive codes include the Hamming codes, various Golay codes, 
certain Hadamard codes, the Nordstrom Robinson codes, certain permutation codes and frequency permutation arrays, which have connections with
powerline communication, and also completely transitive codes, a subfamily of
completely regular codes, which themselves have attracted a lot of interest. It is known that for any neighbour transitive code 
with minimum distance at least 3 there exists a subgroup of $X$ that has a $2$-transitive
action on the alphabet over which the code is defined. Therefore, by Burnside's theorem, this action is of almost simple or affine type.
If the action is of almost simple type, we say the code is \emph{alphabet almost simple neighbour transitive}.
In this paper we characterise a family of neighbour transitive codes, in particular, the alphabet almost simple neighbour
transitive codes with minimum distance at least $3$, and for which the group $X$ has a non-trivial intersection with
the base group of $\Aut(\Gamma)$. If $C$ is such a code, we show that, up to equivalence, there exists a subcode $\Delta$ that 
can be completely described, and that either $C=\Delta$, or $\Delta$ is a neighbour transitive frequency permutation array and
$C$ is the disjoint union of $X$-translates of $\Delta$.

We also prove that any finite group can be identified in a natural way with a neighbour transitive code.
\end{abstract}

\thanks{{\it Date:} draft typeset \today
 }

\maketitle
\section{Introduction}

The Hamming graph $\Gamma=H(m,q)$ has vertex set $V(\Gamma)$ consisting of all $m$-tuples over an
alphabet $Q$ of size $q$ with an edge between two vertices if and only if they differ in exactly one entry.  
The automorphism group of $\Gamma$ is the semi-direct product 
$\Aut(\Gamma)=B\rtimes L$, where $B\cong S_q^m$ and $L\cong S_m$. Any block code $C$
of length $m$ over $Q$ can be embedded as a subset of $V(\Gamma)$.  Two codes in $\Gamma$ are
\emph{equivalent} if there exists an automorphism of $\Aut(\Gamma)$ that maps one to the other.  The 
distance between two vertices in $\Gamma$ is the number of entries in which they differ, and the \emph{minimum 
distance} of $C$ is the minimum of all distances between distinct codewords of $C$.

A \emph{neighbour of $C$} is a vertex that is adjacent to some codeword but is not a codeword itself. We 
denote the set of neighbours of $C$ by $C_1$.  A code $C$ is \emph{$X$-neighbour transitive}, or simply 
\emph{neighbour transitive}, if there exists an $X\leq\Aut(\Gamma)$ such that both $C$ and $C_1$ are $X$-
orbits. If $C$ is an $X$-neighbour transitive code with minimum distance $\delta\geq 3$, 
then there exists a subgroup of $X$ that has a $2$-transitive action on $Q$ \cite[Prop. 2.9]{comtran}, which is 
therefore either of affine type or almost simple type.  If this action is of almost simple type, we say 
$C$ is \emph{alphabet almost simple} $X$-neighbour transitive.  A \emph{frequency permutation array} is a 
code $C$ in $H(m,q)$ where $m=pq$ for some positive integer $p$ such that in each codeword of $C$ every 
letter in the alphabet $Q$ appears exactly $p$ times.  In this paper we characterise the following family 
of neighbour transitive codes.  

\begin{theorem}\label{main} Let $C$ be an alphabet almost simple $X$-neighbour transitive 
code with $K:=X\cap B\neq 1$ and minimum
distance $\delta\geq 3$.  Also let $\soc(K)$ denote the socle of $K$, the group generated by the minimal normal subgroups
of $K$.  Then, up to equivalence, there exists a $\soc(K)$-orbit $\Delta$ of $C$ that is explicitly known, and is described 
in \eqref{des1}, \eqref{des2}, or \eqref{des3}.
Moreover, either $\Delta=C$, or $\Delta$ is a 
neighbour transitive frequency permutation array and $C$ is the disjoint union of $X$-translates of $\Delta$.  
\end{theorem}

This main result, stated in slightly different terms, can be found in the first author's Ph.D thesis \cite[Thm. 6.2]{ngthesis}. 
In Section \ref{secdefprelim} we introduce the necessary definitions along with some preliminary results.  
Then we give various constructions of neighbour transitive codes in Section \ref{secconstr}.
These include the Product and Repetition constructions (Section \ref{secprodrep}) and
the Projection codes (Section \ref{projsec}).  In the next section we describe some concrete examples of neighbour
transitive codes, and in particular, we prove Theorem \ref{fgntthm} (see Section \ref{fgsec}), which 
identifies any finite group with a neighbour transitive code.
We then turn our attention to proving Theorem \ref{main} for an alphabet almost simple $X$-neighbour transitive
code $C$.  First in Section \ref{secstruct} we prove that, up to equivalence, $\soc(K)$ is a sub-direct product of a direct product of non-abelian 
simple groups (Proposition \ref{subdir}).
This allows us to apply Scott's Lemma \cite{scotts}, and in turn, determine the structure of $\soc(K)$.  Using this, in the following section 
we consider certain Projection codes of $C$, and use the classification of diagonally neighbour transitive codes  \cite{diagnt} (see Definition \ref{ntrans}) to
describe the codes that can appear.  This allows us in Section \ref{secbuildingblocks} to describe certain $\soc(K)$-orbits in $C$. 
We prove Theorem~\ref{main} in Section \ref{secproofofthm}. In the final section we given an example of a code that satisfies the conditions
of Theorem \ref{main}, but with the additional property that the projected code has a minimum distance strictly less the minimum distance of the original code.  
In the remainder of this section we discuss the context of our investigation, in particular how it 
relates to earlier studies of completely regular codes, and to recent work on codes suitable for powerline communication.

\subsection{The assumption that errors are equally likely}\label{secerrors}
An ideal decoding decision scheme for communicating across a noisy channel can depend on the 
probability characteristics of the input \cite[p.93]{roman}.  To combat this, it is often assumed that
the input distribution is uniform, that is, each codeword has an equal probability of being sent 
\cite[p.10]{lincostello}.  
In this circumstance, maximum likelihood decision schemes are suitable 
\cite[p.95]{roman}; given an output message $\bf{y}$, the codeword $\bf{x}$ is chosen that maximises 
$p(\bf{y}|\bf{x})$, the probability that $\bf{y}$ was received given that $\bf{x}$ was sent. 
A more descriptive decision scheme is \emph{nearest neighbour decoding} (also known as \emph{minimum
distance decoding}); given a received message $\bf{y}$, a minimum distance decoder will decode to the 
codeword $\bf{x}$ that is closest to $\bf{y}$ with respect to the distance metric inherent in the Hamming 
graph. \emph{Syndrome decoding} used for linear codes is essentially nearest neighbour
decoding, but the algebraic structure of the code allows for any received vertex to be checked against a 
shortened list of codewords \cite[Sec. 1.11]{huffpless}.  

If the probability of an error occurring during transmission is independent of the symbol sent, and also, 
in the case where an error occurs, each of the other $q-1$ symbols has an equal chance of appearing, 
then maximum likelihood decoding is equivalent to nearest neighbour 
decoding \cite[p.130] {roman}, \cite[p.5]{hill}. Thus, if we are transmitting across a discrete 
memoryless channel (the case for which Shannon's Theorem \cite{shannon} was originally proved), assuming 
that single errors are equally likely allows one to use nearest neighbour decoding to obtain maximum 
likelihood decoding.  Hence, in coding theory it is often assumed that the probability of each error occurring is independent of
both the position in which the error occurs, and the symbol appearing in error \cite[p.4]{peterson}, \cite[p.122]{roman}, \cite[p.5]{pless}.  
The authors introduced the property of neighbour transitivity as a group theoretic analogue to the assumption that single errors are equally likely \cite{ntransintro},
and also characterised $X$-neighbour transitive codes for certain groups $X$ \cite{diagnt}.
This paper further contributes to the problem of characterising neighbour transitive codes in general.  

\subsection{Some historical remarks} Ever since Shannon's seminal 1948 paper, coding theorists have been interested in codes that are highly
structured and symmetrical, the hope being that codes with these properties will have good error
correcting capabilities, and at the same time, can be efficiently decoded.

One of the first families of codes that coding theorists studied were \emph{perfect codes}; a code with minimum distance $\delta$ is 
\emph{perfect} if the spheres of radius $e=\lfloor\frac{\delta-1}{2}\rfloor$ centred on the 
codewords cover the vertices of the Hamming graph.  Trivial examples of perfect codes are the code
containing just one codeword, the whole space, and the binary
\emph{repetition code} where $m$ is odd (see Example \ref{repex}). Non-trivial perfect codes include the Hamming codes and the perfect Golay codes.  
Building on work by van Lint, Tiet{\"a}v{\"a}inen \cite{tiet} proved
that any non-trivial perfect code over a finite field has the same parameters as either a Hamming code
or one of the perfect Golay codes.

Once the parameters of perfect codes had been classified, coding theorists began examining 
other families of codes with large amounts of structure, including nearly perfect and uniformly packed codes.
For uniformly packed codes, the spheres of radius $e+1$ around codewords cover the full space, but overlap
in a regular way.  That is, a code is \emph{uniformly packed} if vertices that are at distance $e$ from some codeword are
in $\lambda+1$ spheres, and vertices that are at distance $e+1$ or more from every codeword
are in $\mu$ spheres (where $\lambda$ and $\mu$ are constants and $\lambda<(m-e)(q-1)/(e+1)$).
If $\lambda=\lfloor(m-e)(q-1)/(e+1)\rfloor$ and $\mu=\lfloor m(q-1)/(e+1)\rfloor$, the code is \emph{nearly perfect}.  

Lindstr{\"o}m \cite{lind1} classified the parameters of binary nearly perfect codes; the parameters are those of either 
the punctured Preparata code, or the code constructed by 
puncturing the codewords of even weight in the binary Hamming code.  He also showed that nearly perfect
codes over non-binary finite fields are necessarily perfect \cite{lind2}.  
In his thesis, van Tilborg \cite{vantil} proved the non-existence of uniformly packed codes with $e\geq 4$, and that 
the extended binary Golay code is the only
binary uniformly packed code with $e=3$.  He also 
classified binary uniformly packed codes with $\lambda$ and
$\mu$ such that $\mu-\lambda=1$ (such codes are called \emph{strongly
  uniformly packed}).  A classification of binary linear uniformly
packed codes with $e=2$ was given by Calderbank and Goethals
\cite{cald1}, with one outstanding case being dealt with by Calderbank
\cite{cald2}.

In 1973 Delsarte \cite{delsarte} introduced \emph{completely regular codes}, a family of codes with a high 
degree of combinatorial symmetry (see Definition \ref{sregdef}).
Delsarte showed that perfect codes are completely regular, and it also
holds that uniformly packed codes are completely regular, see for
example \cite{sole}.  Further examples of completely regular codes 
are the Preparata codes, the Kasami codes,
various codes constructed from the Golay codes, and the Hadamard $12$
and punctured Hadamard $12$ code. However, a belief began to emerge amongst coding
theorists that completely regular codes are actually quite rare \cite{martin}.
Indeed, in his 1990 paper, Neumaier \cite{neu} conjectured that
the only completely regular codes with minimum distance at least $8$ are the binary
repetition code and the extended binary Golay code.  
Surprisingly Borges, Rif{\`a} and Zinoviev \cite{conjecex} found a counter 
example to this conjecture, 
and have since written a series of papers classifying various families and finding new examples
of completely regular codes \cite{rho=2,binctrarb,kronprod,liftperfect,extendconst}.  In particular, many of their examples are also 
\emph{completely transitive} (see Definition \ref{ntrans}), a family of completely regular
codes that are very symmetric from an algebraic viewpoint.  Currently the classification
of completely regular codes is an open problem.  A result of Brouwer et al. \cite[p.353]{distreg} shows
that certain families of distance-regular graphs are coset graphs of additive completely regular codes, and so 
a classification of these codes may be of interest to graph theorists as well as coding theorists.  
With this in mind, a classification of completely transitive codes seems valuable, and as such the authors 
have classified all $X$-completely transitive codes with $K:=X\cap B=1$ and minimum distance at least 5 \cite{comtran}.  The case $K\neq 1$ is more difficult, 
however.  As completely transitive codes are necessarily neighbour transitive, any characterisation
of neighbour transitive codes is certainly useful in the context of classifying completely transitive codes.  

\begin{remark} Many of the classifications mentioned above only hold for $q$-ary codes with $q$ a prime power.  
In particular, if $q$ is not a prime power, not even perfect codes have been classified.  There are some results though.  For example, 
Best \cite{best} showed that a perfect code over a non-prime power alphabet must have error correcting capability $e=1,2,6,$ or $8$.
We note that the classification of completely transitive codes with $K=1$ and $\delta\geq 5$ is over any alphabet size.
\end{remark}

\subsection{Connections with powerline communication}\label{powerline}
Powerline communication has been proposed as a possible solution to the 
``last mile problem" in the delivery of telecommunications \cite{hanvinck,stateoftheart}.  
By allowing the frequency at which 
electricity is transmitted over powerlines to vary, 
say $q$ distinct frequencies, an alphabet is generated over which 
information can be encoded \cite{chu3,ferr}.  However, it is likely that 
the power output will not be constant if an arbitrary code is used, interfering
with the primary purpose of the electrical infrastructure.  There are also additional
types of noise that need to be considered for powerline communication.  As well as the usual 
background noise, there is a permanent narrow band noise present generated by electrical equipment
and a short term impulse noise that affects many frequencies over a short period of time~\cite{chu2}.

Constant composition codes have been suggested as suitable coding schemes to deal with
the extra noise considerations present in powerline communication
while at the same time providing a constant power output \cite{chu1,chu2,chu3}.
These are $q$-ary codes with the property that the number of occurrences of each symbol within a codeword
is the same for each codeword. Frequency permutation arrays are a class of constant composition
codes that are ideal, in some sense, for powerline communication \cite{chu2}.  The most extensively
studied frequency permutation arrays (indeed constant composition codes) are 
permutation codes, see for example \cite{bailey,blake,blakeetal,chu1,chu3,frankl,keevash,smith2,slepian}.

It follows from Theorem \ref{main} that frequency permutation arrays arise naturally as the building 
blocks of certain neighbour transitive codes.  In particular, in Section \ref{secbuildingblocks} we see that 
certain neighbour transitive permutation codes are the building blocks in one case, 
and in a second case, twisted permutation codes (which are also frequency permutation 
arrays) are the building blocks.  Twisted permutation codes were introduced by the authors in 
\cite{twisted} and are less well known, but can have improved error correcting capabilities over repeated 
copies of the usual permutation codes \cite[Theorem 1.1]{twisted}.  

\subsection{Every finite group can be considered as a neighbour transitive code}\label{fgsec}

Let $G$ be a finite group of order $q$, and consider the Hamming graph $\Gamma=H(q,q)$ over $G$ of length $q$, 
that is $V(\Gamma)$ consists of all $q$-tuples with entries from $G$. 
Cayley's theorem tells us that $G$ has a faithful regular action on itself by multiplication on the right. 
Let $r(g)$ denote the image of $g\in G$ under this action, and consider
a fixed ordering $\mathfrak{o}=(g_1,\ldots,g_q)$ of the elements of $G$.
Then we can define the following vertex and code in $H(q,q)$.  For $g\in G$ define 
\[\alpha_{\mathfrak{o}}(g)=(g_1^{r(g)},\ldots,g_q^{r(g)})=(g_1g,\ldots,g_qg)\]
and the \emph{permutation code of $G$ with respect to $\mathfrak{o}$} to be 
\begin{equation*}\label{fgpdef}C_{\mathfrak{o}}(G)=\{\alpha_{\mathfrak{o}}(g)\,|\,g\in G\}.\end{equation*}
Given two orderings $\mathfrak{o}$ and $\mathfrak{o}$' of the elements of $G$, we prove in Example \ref{fgex}
that the codes $C_{\mathfrak{o}}(G)$ and $C_{\mathfrak{o}'}(G)$ are equivalent in $\Gamma$.
Thus we can talk of \emph{the permutation code of $G$}, which we denote by $C(G)$.
Let $S_G$ denote the Symmetric group of $G$.  We prove the following theorem in Example \ref{fgex}.

\begin{theorem}\label{fgntthm} Let $G$ be finite group and $C(G)$ be the permutation code of $G$.  Then 
$C(G)$ is $(S_G\times S_G)$-neighbour transitive.
\end{theorem}

\section{Definitions and Preliminaries}\label{secdefprelim}

Recall that any code of length $m$ over an alphabet $Q$ of size $q$ can be
embedded in the vertex set of the \emph{Hamming graph}.  The Hamming graph $\Gamma=H(m,q)$ 
has vertex set $V(\Gamma)$, the set of $m$-tuples with entries from $Q$, and an edge exists between two
vertices if and only if they differ in precisely one entry.
Throughout we assume that $m,q\geq 2$.  The automorphism group of $\Gamma$, which we
denote by $\Aut(\Gamma)$, is the semi-direct product
$B\rtimes L$ where $B\cong S_q^m$ and $L\cong S_m$, see \cite[Theorem 9.2.1]{distreg}.    Let
$g=(g_1,\ldots, g_m)\in B$, $\sigma\in L$ and
$\alpha=(\alpha_1,\ldots,\alpha_m)\in V(\Gamma)$. Then $g$ and $\sigma$
act on $\alpha$ in the following way: \begin{equation}\label{autgpaction}
\alpha^g=(\alpha_1^{g_1},\ldots,\alpha_m^{g_m}),\,\,\,\,\,\,\,\alpha^\sigma=(\alpha_{1\sigma^{-1}},\ldots,\alpha_{m\sigma^{-1}}). \end{equation}
For any subgroup $T$ of $S_q$, we define the following subgroup of
$B$: \begin{equation}\label{diagdef}\Diag_m(T)=\{x_h=(h,\ldots,h)\in B\,:\,h\in T\}.\end{equation}  Let 
$M=\{1,\ldots,m\}$, and view $M$ as the set of vertex entries of
$H(m,q)$.  For all pairs of vertices $\alpha,\beta\in V(\Gamma)$, the \emph{Hamming distance} between
$\alpha$ and $\beta$, denoted by $d(\alpha,\beta)$, is defined to be
the number of entries in which the two vertices differ.  We let
$\Gamma_k(\alpha)$ denote the set of vertices in $H(m,q)$ that are at
distance $k$ from $\alpha$.  

For a code $C$ in $H(m,q)$, the \emph{minimum distance, $\delta$,
  of C} is the smallest distance between distinct codewords of $C$, which we sometimes denote
by $\delta(C)$ if we want to make specific reference to the code.
For any $\gamma\in V(\Gamma)$, we define 
\[d(\gamma,C)=\min\{d(\gamma,\beta)\,:\,\beta\in C\}\] to be the 
\emph{distance of $\gamma$ from $C$}.  The \emph{covering radius of $C$}, 
which we denote by $\rho$, is the maximum distance that any vertex in $H(m,q)$ is from $C$.  
We let $C_i$ denote the set of vertices that are at distance $i$ from $C$, and deduce,
for $i\leq \lfloor (\delta-1)/2\rfloor$, that $C_i$ is the disjoint 
union of $\Gamma_i(\alpha)$ as $\alpha$ varies over $C$.  
Furthermore, $C_0=C$ and $\{C,C_1,\ldots,C_\rho\}$ forms a
partition of $V(\Gamma)$ called the \emph{distance partition of $C$}. 
In particular, the \emph{complete code} $C=V(\Gamma)$ has covering
radius $0$ and trivial distance partition $\{C\}$; and if $C$ is not
the complete code, we call the non-empty subset $C_1$ the \emph{set of
  neighbours of $C$}. 

\begin{definition}\label{ntrans} Let $C$ be a code with distance partition
  $\{C,C_1,\ldots,C_\rho\}$.  Recall that $C$ is \emph{$X$-neighbour
    transitive}, or simply \emph{neighbour transitive}, if there exists $X\leq \Aut(\Gamma)$ such that $C$ and
  $C_1$ are $X$-orbits in $H(m,q)$.  If $C$ is $X$-neighbour transitive with $X\leq\Diag_m(S_q)\rtimes L$ then we say $C$
is \emph{diagonally neighbour transitive}.  If each $C_i$ is an $X$-orbit for $i=0,\ldots,\rho$ we say $C$ is \emph{$X$-completely transitive},
or simply \emph{completely transitive}.
\end{definition}

\begin{example}\label{repex} For $a\in Q$ let $(a^m)=(a,\ldots,a)\in H(m,q)$.  The \emph{repetition code} in $H(m,q)$ is the code
  \[\Rep(m,q)=\{(a^m)\,:\,a\in Q\}.\] It has minimum distance $\delta=m$ and is one of the simplest neighbour transitive codes \cite{diagnt}.  
  It is also true that $\Rep(m,2)$ is completely transitive \cite{comtran}.  However,
 if $m\geq 4$ and $q\geq 3$ then $\Rep(m,q)$ is not completely transitive \cite[Lemma 2.15]{comtran}.
\end{example}

We say two codes $C$ and $C'$ are \emph{equivalent} if there exists $x\in\Aut(\Gamma)$
such that $C^x=C'$, and if $C'=C$ we call $x$ an automorphism of $C$.  
The \emph{automorphism group of $C$}, denoted by $\Aut(C)$, is the setwise stabiliser of $C$ in $\Aut(\Gamma)$. It turns out
that any neighbour transitive code in $H(m,q)$ with minimum distance $\delta=m$ is equivalent to $\Rep(m,q)$.  To explain this result, we 
introduce \emph{$s$-regular} codes. 

\begin{definition}\label{sregdef} A code with covering radius $\rho$ is \emph{$s$-regular}, for a given $s\in\{0,\ldots,\rho\}$, 
if for $k\geq 1$ and $\nu\in C_i$, with $0\leq i\leq s$, the cardinality of the set $\Gamma_k(\nu)\cap C$ is independent of $\nu$ 
and only depends on $i$ and $k$.  A code that is $\rho$-regular is said to be \emph{completely regular.}\end{definition}
\noindent It is a consequence of the definitions that neighbour transitive codes are necessarily $1$-regular, and
it is known that completely transitive codes are completely regular \cite{giudici}.  The next result follows directly from 
\cite[Lemma 2.13]{comtran}.

\begin{lemma}\label{cardc=q}  Let $C$ be a code in $H(m,q)$ with $|C|\geq 2$ and
  $\delta=m\geq 2$.  Then there exists $C'$ equivalent to $C$ with
  $C'\subseteq\Rep(m,q)$.  Moreover if $C$ is $1$-regular then
  $C'=\Rep(m,q)$.  In particular, any neighbour transitive code with $\delta=m$ is equivalent to $\Rep(m,q)$.
\end{lemma}

Let $C$ be a code and $\alpha$ be any vertex in $H(m,q)$.  As $\Aut(\Gamma)$ acts transitively on
the vertices of $\Gamma$, there exists $y\in \Aut(\Gamma)$ such that $\alpha\in C^y$. The next result allows us
to take advantage of this fact.

\begin{lemma}([Lem. 2, \cite{diagnt}])\label{dpartntreq} Let $C$ be a code with distance partition
  $\mathcal{C}=\{C,C_1,\ldots,C_\rho\}$ and $y\in\Aut(\Gamma)$.  Then
  $C_i^y:=(C_i)^y=(C^y)_i$ for each $i$.  In particular, the code
  $C^y$ has distance partition $\{C^y,C_1^y,\ldots,C^y_\rho\}$, and
  $\mathcal{C}$ is $\Aut(C)$-invariant.  Moreover, $C$ is
  $X$-neighbour transitive if and only if $C^y$ is
  $X^y$-neighbour transitive.      
\end{lemma}

\noindent By replacing $C$ with the equivalent code $C^y$ if necessary, Lemma \ref{dpartntreq} allows us to assume 
that $\alpha$ is a codeword in our neighbour transitive code $C$. We use this trick several times throughout this paper.

\subsection{Description of Neighbours}\label{descneigh} Let
$\alpha=(\alpha_1,\ldots,\alpha_m)$ be a vertex in $H(m,q)$, and for 
$a\in Q$ let $\nu(\alpha,i,a)$ denote the vertex with $j$th entry 
\begin{equation}\label{neiform}\nu(\alpha,i,a)|_j=\left\{\begin{array}{ll}     
 \alpha_j&\textnormal{if $j\neq i$}\\ a &\textnormal{if
   $j=i$.} \end{array}\right.\end{equation}
We note that if $\alpha_i=a$ then $\nu(\alpha,i,a)=\alpha$, otherwise it is adjacent to $\alpha$.  
Throughout this paper if we refer to $\nu(\alpha,i,a)$ as a neighbour of $\alpha$, or being
adjacent to $\alpha$, the reader should assume that $a\in Q\backslash\{\alpha_i\}$.
For $x=(h_1,\ldots,h_m)\sigma\in\Aut(\Gamma)$ it is known that 
\begin{equation}\label{neiact2} \nu(\alpha,i,a)^x=\nu(\alpha^x,i^\sigma,a^{h_i}),\end{equation}
which is neighbour of $\alpha^x$ if and only if $\nu(\alpha,i,a)$ is a neighbour of $\alpha$ \cite[Lemma 1]{diagnt}.
Combining this with the following result will prove useful in the sequel.

\begin{lemma}\label{little} Let $\nu(\alpha,i,a)$ and $\nu(\beta,j,b)$
  be respective neighbours of $\alpha=(\alpha_1,\ldots,\alpha_m)$ and  
  $\beta=(\beta_1,\ldots,\beta_m)$ in $H(m,q)$ such that
  $\nu(\alpha,i,a)=\nu(\beta,j,b)$.  Then one of the following holds: 
\begin{itemize}
\item[(i)] $\alpha=\beta$, $i=j$ and $a=b$;
\item[(ii)]$i=j$, $a=b$ and $\beta=\nu(\alpha,i,c)$ for some $c\neq \alpha_i,a$;
\item[(iii)] $d(\alpha,\beta)=2$, $i\neq j$, $\alpha_j=b$ and $\beta_i=a$.
\end{itemize}
\end{lemma}

\begin{proof} It follows that $d(\alpha,\beta)\leq
  d(\alpha,\nu(\alpha,i,a))+d(\nu(\beta,j,b),\beta)=2$ since
  $\nu(\alpha,i,a)=\nu(\beta,j,b)$.  Firstly assume $d(\alpha,\beta)=0$.
  Then $\alpha=\beta$.  Now suppose $i\neq j$.  Then
  $\nu(\alpha,i,a)|_i=\nu(\beta,j,b)|_i=\alpha_i$ since $\alpha=\beta$
  and $i\neq j$.  Thus $\alpha=\nu(\alpha,i,a)$ contradicting the fact that
  $\nu(\alpha,i,a)$ is a neighbour of $\alpha$.  Hence $i=j$ and so
  $\nu(\alpha,i,a)|_i=a=b=\nu(\beta,i,b)|_i$.  Thus (i) holds.  
  
  Now assume $d(\alpha,\beta)=1$.  Then $\beta=\nu(\alpha,k,c)$ for
  some $c\neq\alpha_k$.  Suppose $i,j\neq k$.  Then
  $\nu(\alpha,i,a)|_k=\alpha_k=\nu(\beta,j,b)|_k=c$, which is a
  contradiction.  Thus $i=k$ or $j=k$.  Suppose $i=k$ and $j\neq k$.
  Then $b\neq\beta_j=\alpha_j$.  However, 
  $\nu(\alpha,k,a)|_j=\alpha_j=\nu(\beta,j,b)_j=b$, which is a
  contradiction.  The condition $j=k$ and $i\neq k$ leads to a
  similar contradiction.  Thus $i=j=k$.  Hence
  $\nu(\alpha,i,a)|_i=a=b=\nu(\beta,i,b)|_i$.  Suppose $c=a$ and so
  $c=b$.  Then $\nu(\beta,j,b)=\nu(\beta,k,c)=\beta$, contradicting the fact that
  $\nu(\beta,j,b)$ is a neighbour of $\beta$.  Thus (ii) holds.
  
  Finally, assume $d(\alpha,\beta)=2$.  Suppose $i=j$.  Then
  $\nu(\alpha,i,a)=\nu(\beta,i,b)$, and so by definition,
  $\alpha_k=\beta_k$ for all $k\neq i$.  Thus $d(\alpha,\beta)\leq 1$
  and we have a contradiction.  Thus $i\neq j$.  It follows from the
  definitions that $\alpha_j=\nu(\beta,j,b)=b$ and $\beta_i=\nu(\alpha,i,a)=a$. 
\end{proof}
 
\subsection{Group Actions}\label{grpactsec}
For a nonempty set $\Omega$, we denote the group of permutations of $\Omega$ by $\Sym(\Omega)$.  A 
\emph{permutation group} $G$ on $\Omega$ is a subgroup of $\Sym(\Omega)$.  
The \emph{minimal degree} of $G$ is the smallest number of points moved by any non-identity element of $G$.
We say $G$ acts \emph{regularly} on $\Omega$ if $G$ is a transitive subgroup of $\Sym(\Omega)$ and $G_\alpha=1$ for all $\alpha\in\Omega$.  

For an abstract group $G$,  an \emph{action} of $G$ on $\Omega$ is a homomorphism $\rho$ from $G$ to $\Sym(\Omega)$. 
The \emph{degree of the action} is the cardinality of $\Omega$.  
Let $\rho_1:G\longrightarrow\Sym(\Omega)$ and $\rho_2:H\longrightarrow\Sym(\Omega')$ be actions 
of the groups $G,H$ on $\Omega, \Omega'$ respectively.  We say these actions are \emph{permutationally isomorphic} if there exists a 
bijection $\lambda:\Omega\longrightarrow\Omega'$ and an
isomorphism $\varphi:\rho_1(G)\longrightarrow \rho_2(H)$ such that 
\begin{equation*}\lambda(\alpha^{\rho_1(g)})=\lambda(\alpha)^{\varphi(\rho_1(g))}\quad\quad
\textnormal{for all $\alpha\in\Omega$ and $g\in G$},\end{equation*}
and we call $(\lambda,\varphi)$ a \emph{permutational isomorphism}.  

We now consider three distinct actions
for the automorphism group $X$ of an $X$-neighbour transitive code.  First we consider its natural
action on the code.

\begin{lemma}\label{codeimprim}  Let $C$ be an $X$-neighbour transitive code with
  minimum distance $\delta\geq 3$.  Let $\Delta$ be a block of imprimitivity for the
  action of $X$ on $C$.  Then $\Delta$ is an $X_{\Delta}$-neighbour
  transitive code with minimum distance at least $\delta$.
\end{lemma}

\begin{proof} Since $\Delta$ is a block of imprimitivity for the
  action of $X$ on $C$, it follows that $X_\Delta$ acts transitively on $\Delta$ \cite[p.13]{dixonmort}.  
  Let $\nu_1,\nu_2$ be neighbours of $\Delta$.  Then there exist codewords $\alpha_1,
  \alpha_2$ of $\Delta$ that are respectively adjacent to $\nu_1,
  \nu_2$.  As $C$ is $X$-neighbour transitive, there exists $x\in X$ such that $\nu_1^x=\nu_2$.  We claim
  that $\alpha_1^x=\alpha_2$.  Suppose not.  Then there exists a codeword
  $\alpha_3\in C$ such that $\alpha_1^x=\alpha_3$ and $\nu_2$ is a
  neighbour of $\alpha_3$.  This implies that
  $d(\alpha_2,\alpha_3)\leq 2$, contradicting the minimum distance of
  $C$.  Hence $\alpha_1^x=\alpha_2$. Therefore, because $\Delta$ is a
  block of imprimitivity, $\Delta^x=\Delta$.  Finally, because 
  $C$ has minimum distance $\delta$ and $\Delta\subseteq C$, it follows
  directly that $\Delta$ has minimum distance at least $\delta$.
\end{proof}

Let us now describe two alternative actions for an automorphism group $X\leq \Aut(\Gamma)$.  We define \[\begin{array}{c c c c}   
     \mu:&\Aut(\Gamma)&\longrightarrow&S_m\\ 
&h\sigma&\longmapsto&\sigma,\\
\end{array}
\]
and when we talk of the action of $X$ on $M$, or the action of $X$ on
entries, we are referring to the action of $\mu(X)$ on $M$.  We denote $\mu(X)$ by $X^M$.  
Now, for $i\in M$ let $X_i=\{x=h\sigma\in X\,:\,i^\sigma=i\}$, which has an action on the alphabet $Q$ 
via the following homomorphism:    
\begin{equation}\label{varhom}\begin{array}{c c c c}
     \varphi_i:&X_i&\longrightarrow&S_q\\ 
&(h_1,\ldots,h_m)\sigma&\longmapsto&h_i\\
\end{array}\end{equation}  We denote the image $\varphi_i(X_i)$ by $X_i^Q$. Next we collate three results that appear in 
\cite[Prop. 2.7, Cor. 2.8, Prop. 2.9]{comtran}, applying them to our context.  

\begin{proposition}\label{actionsprop} Let $C$ be an $X$-neighbour
  transitive code with $\delta\geq 3$.  Let $\alpha\in C$ and $i\in M$. Then
\begin{itemize}
\item[(i)] $X_\alpha$ acts transitively on $\Gamma_1(\alpha)$ and $M$,
\item[(ii)] $X_i$ acts transitively on $C$, and 
\item[(iii)] $X_i^Q$ acts $2$-transitively.    
\end{itemize} 
\end{proposition}

Proposition \ref{actionsprop} gives us that $X_1^Q$ is $2$-transitive for
any $X$-neighbour transitive codes with minimum distance $\delta\geq 3$.  It
follows from Burnside's Theorem \cite[Theorem 4.1B]{dixonmort} that any $2$-transitive group is of
almost simple type or of affine type.  Thus, we recall the following definition from the introduction.

\begin{definition}\label{alphaasaf} Let $C$ be an $X$-neighbour
  transitive code.   We say $C$ is \emph{alphabet
    almost simple neighbour transitive} if $X^M$ acts transitively on $M$
  and $X_1^Q$ is $2$-transitive of almost simple type.  Similarly, if $X^M$ acts transitively
  and $X_1^Q$ is $2$-transitive of affine type, we say $C$ is \emph{alphabet affine neighbour 
    transitive.}    
\end{definition}

\section{Constructions of Neighbour Transitive Codes}\label{secconstr}

\subsection{Product and Repetition Constructions}\label{secprodrep}

In this section we consider $\ell$-tuples of codewords from a code $C$ in $H(m,q)$.
Let us first consider the set of all $\ell$-tuples of vertices from $H(m,q)$, that is, \[\Gamma^{\ell}=\{(\alpha_1,\ldots, \alpha_\ell)\,:\,\alpha_i\in H(m,q)\}.\]
It is clear that we can identify $\Gamma^{\ell}$ with the vertex set of $H(m\ell,q)$.  
For any $X\leq\Aut(\Gamma)$ we define an action of $X\wr S_\ell$ on $\Gamma^\ell$ in the natural way.  In particular, for 
$\bs{\alpha}=(\alpha_1,\ldots,\alpha_\ell)\in\Gamma^{\ell}$ and $y=(x_1,\ldots,x_\ell)\sigma\in X\wr S_\ell$, 
\begin{equation}\label{prodact}
\bs{\alpha}^{y}=(\alpha^{x_1}_{1},\ldots,\alpha^{x_\ell}_{\ell})^\sigma=
(\alpha_{1\sigma^{-1}}^{x_{1\sigma^{-1}}},\ldots,\alpha^{x_{\ell\sigma^{-1}}}_{\ell\sigma^{-1}}).\end{equation}
Now, for an arbitrary code $C$ in $H(m,q)$ we denote the complete set of
$\ell$-tuples of codewords from $C$ 
by \[\Prod_\ell(C)=\{(\alpha_1,\ldots,\alpha_\ell)\,|\,\alpha_i\in 
C\textnormal{ for all $i$ }\}.\]
For $\bs{\alpha}\in\Gamma^{\ell}$ and $\nu\in\Gamma$ we let $\gamma(\bs{\alpha},i,\nu)$ denote
the vertex constructed by changing that $i$th vertex entry of $\bs{\alpha}$ from $\alpha_i$ to $\nu$.  
It follows that $\gamma(\bs{\alpha},i,\nu)$ is in $\Prod_\ell(C)_1$, the neighbour set of $\Prod_\ell(C)$,
if and only if $\nu\in C_1$. Also, all members of $\Prod_\ell(C)_1$ are of this form.

\begin{lemma}\label{prodconstnt}  Let $C$ be an $X$-neighbour transitive code in $H(m,q)$ with minimum 
distance $\delta$.  Then $\Prod_\ell(C)$ is ($X\wr S_\ell$)-neighbour
transitive in $H(m\ell,q)$ with minimum distance $\delta$.  
\end{lemma}

\begin{proof} It is clear that $X\wr S_\ell$ fixes $\Prod_\ell(C)$ setwise, and so by Lemma \ref{dpartntreq}, it 
also fixes setwise the neighbour set of $\Prod_\ell(C)$.  Moreover, because $X$ acts transitively on $C$
it follows that $X^\ell$ acts transitively of $\Prod_\ell(C)$.  Now let $\gamma(\bs{\alpha},i,\nu)$ and $\gamma(\bs{\beta},j,\nu')$
be two neighbours of $\Prod_\ell(C)$.  As $S_\ell$ acts transitively on $\{1,\ldots,\ell\}$ and is contained in the automorphism group of $\Prod_\ell(C)$, 
we can assume without loss of generality
that $i=j$.  As both $C$ and $C_1$ are $X$-orbits in $H(m,q)$, there exist $x_k\in X$ such that $\alpha_k^{x_k}=\beta_k$ for
each $k\neq i$, and $x_i\in X$ such that $\nu^{x_i}=\nu'$.  By letting $y=(x_1,\ldots,x_\ell)$, it follows that 
$\gamma(\bs{\alpha},i,\nu)^y=\gamma(\bs{\beta},i,\nu')$.
\end{proof}

An interesting subcode of $\Prod_\ell(C)$ is the set of all vertices with
constant entry.   We let \[\rep_\ell(\alpha)=(\alpha,\ldots,\alpha)\in
\Gamma^\ell\] for a vertex $\alpha\in H(m,q)$, and we
define \[\Rep_\ell(C)=\{\rep_\ell(\alpha)\,|\,\alpha\in C\}.\]
It follows from \cite[Lemma 5]{diagnt} and Proposition \ref{actionsprop} that if $C$ is $X$-neighbour
transitive with $\delta\geq 3$, then $\Rep_\ell(C)$ is $(X\times S_\ell)$-neighbour
transitive with minimum distance $\delta\ell$.  

\subsection{Projection Codes}\label{projsec}

Given any code $C$ we want to describe codes that are somehow
``contained'' in $C$ but in a ``smaller'' Hamming graph.  To explain this
idea let $J=\{i_1,\ldots,i_k\}\subseteq
M$, with $i_1<i_2<\ldots<i_k$, and define the following 
map \[\begin{array}{c c c c} \pi_J:&H(m,q)&\longrightarrow&H(J,q)\\  
&(\alpha_1,\ldots,\alpha_m)&\longmapsto&(\alpha_{i_1},\ldots,\alpha_{i_k})\\
\end{array}
\] Note that by $H(J,q)$ we mean the Hamming graph $H(|J|,q)$.  For a code $C$ in $H(m,q)$ we define the \emph{projected code of $C$ with
  respect to $J$} to be the set
\[\pi_J(C)=\{\pi_J(\alpha)\in H(J,q)\,:\,\alpha\in C\}.\] We are interested in 
projected codes of neighbour transitive codes, and therefore we want to obtain some
group information when we project.  Thus for $x\in \Aut(\Gamma)_J=\{h\sigma\in 
\Aut(\Gamma)\,:\,J^\sigma=J\}$ we define $\chi(x)$ to be the 
map \begin{equation}\label{chimap}\begin{array}{c c c c}
\chi(x):&H(J,q)&\longrightarrow&H(J,q)\\  
&\pi_J(\alpha)&\longmapsto&\pi_J(\alpha^x)\end{array}\end{equation}
We observe that this map is well defined if and only if $x\in\Aut(\Gamma)_J$, 
and that $\chi(\Aut(\Gamma)_J)\leq\Aut(H(J,q))$.

\begin{remark}\label{chirem} Let
  $x=(h_1,\ldots,h_m)\sigma\in\Aut(\Gamma)_J$ where
  $J=\{i_1,\ldots,i_k\}$.  Let $\hat{\sigma}$ be the induced action of
  $\sigma$ on $J$.  We claim that
  $\chi(x)=(h_{i_1},\ldots,h_{i_k})\hat{\sigma}$.  Let
  $\alpha=(\alpha_{i_1},\ldots,\alpha_{i_k})\in H(J,q)$.  Then there
  exists $\beta=(\alpha_1,\ldots,\alpha_m)\in H(m,q)$ such that
  $\pi_J(\beta)=\alpha$.  Thus
  $\alpha^{\chi(x)}=\pi_J(\beta)^{\chi(x)}=\pi_J(\beta^x)$, and so 
  \[\chi_J(\beta^x)=(\alpha_{{i_1}{\sigma^{-1}}}^{h_{{i_1}{\sigma^{-1}}}},\ldots,\alpha_{{i_k}{\sigma^{-1}}}^{h_{{i_k}{\sigma^{-1}}}}).\]
  By definition, $i^\sigma=i^{\hat{\sigma}}$ for all $i\in J$.
  Therefore, since $J^\sigma=J$, it follows
  that $\chi_J(\beta^x)=\alpha^{(h_{i_1},\ldots,h_{i_k})\hat{\sigma}}.$
  Because $\alpha$ was arbitrarily chosen, we conclude that
  $\chi(x)=(h_{i_1},\ldots,h_{i_k})\hat{\sigma}$.   
\end{remark}

Let $\pi_J(C)$ be the projected code in $H(J,q)$ of $C$ with distance partition 
$\{\pi_J(C),\pi_J(C)_1,\ldots,\pi_J(C)_{\rho_J}\}$.
In order to examine how the distance partition of $\pi_J(C)$ relates to the distance partition of $C$
we introduce the following set:\[C_1(J)=\{\nu(\alpha,j,b)\in 
C_1\,:\,\textnormal{$\alpha\in C$, $j\in J$ and $b\in Q\backslash\{\alpha_j\}$}\}.\]
\begin{lemma}\label{projneigh} Let $C$ be a code in $H(m,q)$ and $J\subseteq
  M$.  Then 
\begin{itemize}
\item[(i)] for all $\nu\in C_i$, $\pi_J(\nu)\in \pi_J(C)_k$ for some $k\leq i$.  
\item[(ii)] If $\delta\geq 2$, then $\pi_J(C)_1\subseteq\pi_J(C_1(J))$. 
\end{itemize}
\end{lemma}

\begin{proof}
  (i) For $\nu\in C_i$, there exists $\alpha\in C$ such that
  $d(\alpha,\nu)=i$, and so $d(\pi_{J}(\alpha),\pi_J(\nu))\leq i$.
  Hence $d(\pi_J(\nu),\pi_J(C))\leq i$.

  (ii) Let $\nu=(\nu_{i_1},\ldots,\nu_{i_k})\in\pi_J(C)_1$.
  Then $\nu$ is the neighbour of a codeword
  $\pi_J(\alpha)$ for some $\alpha=(\alpha_1,\ldots,\alpha_m)\in C$.  Thus
  there exists $j\in J$ such that $\nu_i=\pi_J(\alpha)_i=\alpha_i$ for all $i\in
  J\backslash\{j\}$ and $\nu_j\neq\pi_J(\alpha)_j=\alpha_j$.  Consider
  the vertex $\nu(\alpha,j,\nu_j)\in H(m,q)$.  Since
  $\nu_j\neq\alpha_j$ it follows that $\nu(\alpha,j,\nu_j)$ is
  adjacent to $\alpha$.  Moreover, $\nu(\alpha,j,\nu_j)\in C_1$ since
  $\delta\geq 2$.  Because $j\in J$, it follows that
  $\nu(\alpha,j,\nu_j)\in C_1(J)$.  In addition
  we have that $\pi_J(\nu(\alpha,j,\nu_j))=\nu$.  Thus
  $\nu\in\pi_J(C_1(J))$.
\end{proof}

We observe that the reverse inclusion of Lemma \ref{projneigh}--(ii)
does not always hold.  For example, let $C$ be a code with $\delta=2$
and $\alpha,\beta\in C$ such that $d(\alpha,\beta)=2$ with 
$\alpha,\beta$ differing in entries $i,k\in M$.  Let $J$ be a proper
subset of $M$ that contains $i$ but not $k$.  Consider the vertex
$\nu=\nu(\alpha,i,\beta_i)$, which is adjacent to $\alpha$.  Since
$\delta=2$ and $i\in J$, it follows that $\nu\in C_1(J)$.  However,
$\pi_J(\nu)=\pi_J(\beta)\in\pi_J(C)$.  We now show that, given certain
conditions, including $\delta\geq 3$, the projected code of a neighbour transitive code is
also neighbour transitive.    

\begin{proposition}\label{projcode}  Let $C$ be an $X$-neighbour transitive code with
  $\delta\geq 3$.  Moreover, let $\mathcal{J}=\{J_1,\ldots,J_\ell\}$ be an
  $X$-invariant partition of $M$.  Then $\pi_{J}(C)$ is a
  $\chi(X_{J})$-neighbour transitive code for each
  $J\in\mathcal{J}$. 
\end{proposition}

\begin{proof}  Let $\alpha\in C$ and $x\in X_J$.  As $\alpha^x\in C$ it follows
  that $\pi_J(\alpha)^{\chi(x)}=\pi_J(\alpha^x)\in\pi_J(C)$.  Hence $\chi(X_J)$ is an automorphism group
  of $\pi_J(C)$.  Now, because $\delta\geq  3$, it follows from Proposition \ref{actionsprop}--(ii)
  that, for $j\in J$, $X_j$ acts transitively on $C$.  Thus, as $\mathcal{J}$ is an
  $X$-invariant partition of $M$, $X_j\leq X_{J}$, and so $X_{J}$ acts
  transitively on $C$.  From this we deduce that $\chi(X_J)$ acts transitively on $\pi_J(C)$.    

  As $\chi(X_J)$ is an automorphism group of $\pi_J(C)$, Lemma \ref{dpartntreq} implies that 
  $\chi(X_J)$ fixes setwise $\pi_J(C)_1$, the set of neighbours of $\pi_J(C)$.  Now let 
  $\nu_1,\nu_2\in\pi_{J}(C)_1$.  By Lemma \ref{projneigh}, 
  there exists $\alpha,\beta\in C$, $i,j\in J$, $b,c\in Q$ such that
  $\nu(\alpha,i,b)$, $\nu(\beta,j,c)\in C_1(J)$ and $\pi_J(\nu(\alpha,i,b))=\nu_1$, $\pi_J(\nu(\beta,j,c))=\nu_2$.  
  Since $C$ is $X$-neighbour
  transitive, there exists $x=h\sigma\in X$ such that
  $\nu(\alpha,i,b)^{h\sigma}=\nu(\beta,j,c)$.  As $\delta\geq 3$, it follows from 
  \eqref{neiact2} and Lemma \ref{little} that $i^\sigma=j$.  Therefore, because 
  $\mathcal{J}$ is an $X$-invariant partition of $M$, it follows that
  $x\in X_{J}$.  Thus \[\nu_1^{\chi(x)}=\pi_{J}(\nu(\alpha,i,b)^x)=\pi_{J}(\nu(\beta,j,c))=\nu_2.\]
  Hence $\pi_{J}(C)$ is $\chi(X_{J})$-neighbour transitive. 
\end{proof}

We saw in Lemma \ref{projneigh}--(ii) that if $\delta\geq
2$ then $\pi_J(C)_1\subseteq \pi_J(C_1(J))$.  We now investigate the
reverse inclusion, supposing that the conditions of Proposition
\ref{projcode} hold.

\begin{lemma}\label{projdistpart} Let $C$ be an $X$-neighbour transitive code with
  $\delta\geq 3$.  Moreover let $\mathcal{J}=\{J_1,\ldots,J_\ell\}$ be
  an $X$-invariant partition of $M$.  Then, for each $J\in\mathcal{J}$, either $\pi_{J}(C)$ is
  the complete code or $\pi_{J}(C)_1=\pi_{J}(C_1(J))$. 
\end{lemma}

\begin{proof}
Let $\nu(\alpha,j,b)\in C_1(J)$ and $x=h\sigma\in X_J$.  Then by \eqref{neiact2}, $\nu(\alpha,j,b)^x=\nu(\alpha^x,j^\sigma,b^{h_j})$.  
As both $C$ and $C_1$ are $X$-orbits, $\nu(\alpha^x,j^\sigma,b^{h_j})\in C_1$ with $\alpha^x\in C$, and because $x\in X_J$, $j^\sigma\in J$.
In particular, $\nu(\alpha^x,j^\sigma,b^{h_j})\in C_1(J)$, that is, $X_J$ stabilises $C_1(J)$.

Now let $\nu(\beta,i,c)\in C_1(J)$.  It follows from the neighbour transitivity of $C$ that there exists $x=h\sigma\in X$ such that $\nu(\alpha,j,b)^x=\nu(\beta,i,c)$.
As $\delta\geq 3$, Lemma \ref{little} implies that $j^\sigma=i$, and so $x\in X_J$.  
Hence $X_J$ acts transitively on $C_1(J)$, from which   
it naturally follows that $\pi_J(C_1(J))$ is a $\chi(X_J)$-orbit.  By Lemma~\ref{projneigh} and Proposition~\ref{projcode}, 
$\pi_J(C)_1$ is a subset of $\pi_J(C_1(J))$ that is also a $\chi(X_J)$-orbit, and so 
we deduce that either $\pi_J(C)_1=\pi_J(C_1(J))$ or $\pi_J(C)_1=\emptyset$.  The latter case holds if and only if $\pi_J(C)$ is the complete code in $H(J,q)$.  
\end{proof}

The next two results give us a lower bound on the minimum distance of a projected code of a neighbour transitive
code for which the conditions of Proposition \ref{projcode} hold.  

\begin{lemma}\label{projneigh2} Let $C$ be an $X$-neighbour transitive code with
  $\delta\geq 3$ and $\mathcal{J}=\{J_1,\ldots,J_\ell\}$ be an $X$-invariant partition of 
  $M$.  Let $J,J'\in\mathcal{J}$.  Then the action of $\chi(X_J)$ on $\pi_J(C)$ is permutationally isomorphic to
the action of $\chi(X_{J'})$ on $\pi_{J'}(C)$.  Moreover, $\delta(\pi_{J}(C))=\delta(\pi_{J'}(C))$.   
\end{lemma}

\begin{proof}
As $C$ is $X$-neighbour transitive with $\delta\geq 3$, it follows from Proposition \ref{actionsprop}--(ii) that
there exists $y=h\sigma\in X$ such that $J^\sigma=J'$. 
Define the map $\lambda_y:\pi_J(C)\longrightarrow\pi_{J'}(C)$ given by $\pi_J(\alpha)\mapsto\pi_{J'}(\alpha^y)$. 
As $y\in\Aut(C)$, it is clear that $\lambda_y$ maps onto $\pi_{J'}(C)$. Moreover, for $\alpha,\beta\in C$,
\begin{align*}
\pi_{J'}(\alpha^y)=\pi_{J'}(\beta^y)&\iff \alpha^y|_k=\beta^y|_k&\forall k\in J'\\
&\iff \alpha_{k\sigma^{-1}}^{h_{k\sigma^{-1}}}=\alpha_{k\sigma^{-1}}^{h_{k\sigma^{-1}}}&\forall k\in J'\\
&\iff\alpha_i^{h_i}=\beta_i^{h_i}&\forall i\in J&\,\,\,\,\,\,\,\textnormal{(as $J^\sigma=J'$)}\\
&\iff \alpha_i=\beta_i&\forall i\in J\\
&\iff\pi_J(\alpha)=\pi_J(\beta),
\end{align*}
so $\lambda_y$ is a bijection from $\pi_J(C)$ to $\pi_{J'}(C)$. The map $\varphi_y:\chi(X_J)\longrightarrow\chi(X_{J'})$ 
given by $\chi(x)\mapsto\chi(y^{-1}xy)$ is an isomorphism, and one can deduce that 
$(\lambda_y,\varphi_y)$ is a permutational isomorphism from the action of $\chi(X_J)$ on $\pi_J(C)$ 
to the action of $\chi(X_{J'})$ on $\pi_{J'}(C)$.  

The argument above shows that $\alpha^y|_k=\beta^y|_k$ for $k\in J'$ if
and only if $\alpha_{k{\sigma^{-1}}}=\beta_{k{\sigma^{-1}}}$ for $k{\sigma^{-1}}\in J$.  Hence 
$d(\pi_{J}(\alpha),\pi_{J}(\beta))=d(\pi_{J'}(\alpha^y),\pi_{J'}(\beta^y))$. Now let $\delta_J=\delta(\pi_{J}(C))$ and
  $\delta_{J'}=\delta(\pi_{J'}(C))$.  By definition, there exist
  $\alpha,\beta\in C$ such
  that $d(\pi_{J}(\alpha),\pi_{J}(\beta))=\delta_J$.
  Therefore $d(\pi_{J'}(\alpha^y),\pi_{J'}(\beta^y))=\delta_J$
  and so $\delta_J\geq\delta_{J'}$.  Similarly it follows that
  $\delta_{J'}\geq \delta_J$.  Hence $\delta_J=\delta_{J'}$.
\end{proof}

\begin{corollary}\label{projcodedist}   Let $C$ be an $X$-neighbour transitive code with
  $\delta\geq 3$ and $\mathcal{J}$ be an $X$-invariant partition of 
  $M$ with $J\in\mathcal{J}$.  If $\pi_J(C)$ is not the complete code then $\delta(\pi_{J}(C))\geq 2$.
\end{corollary}

\begin{proof}  Suppose $\delta(\pi_{J}(C))=1$.  Then there
  exist $\alpha,\beta\in C$ such that $d(\pi_{J}(\alpha),\pi_{J}(\beta))=1$.  In particular, 
  there exists $k\in J$ such that $\alpha_k\neq\beta_k$ and
  $\alpha_j=\beta_j$ for all $j\in J\backslash\{k\}$.  Let $\nu=\nu(\alpha,k,\beta_k)$, which because $\delta\geq 3$ and $k\in J$ is an element of $C_1(J)$.  If $\pi_J(C)$ is not the complete code then Lemma \ref{projdistpart} implies that
  $\pi_{J}(\nu)\in\pi_{J}(C)_1$.  However
  $\pi_{J}(\nu)=\pi_{J}(\beta)\in \pi_{J}(C)$ which is a
  contradiction, so $\delta(\pi_{J}(C))\geq 2$.      
\end{proof}

\section{Examples of neighbour transitive codes}\label{secntrex}
We now give some further examples of neighbour transitive codes.  

\begin{example}\label{injex} Let $m<q$ and \[C=\{(\alpha_1,\ldots,\alpha_m)\in
  H(m,q)\,:\,\textnormal{$\alpha_i\neq\alpha_j$ for $i\neq j$}\}.\]
 Then $C$ has minimum distance $\delta=1$ and covering radius $\rho=m-1$.\end{example}
\begin{example}\label{weightex} If $0$ is a distinguished letter of $Q$, the \emph{weight} of a vertex $\alpha$ is the number of entries
not equal to $0$.  Let $m\geq 3$ be odd and $Q=\{0,1\}$.  Define \[C=\{\alpha\in
  H(m,2)\,:\,\textnormal{$\alpha$ has weight $\frac{m-1}{2}$ or weight
  $\frac{m+1}{2}$}\}.\] Then $C$ has minimum distance $\delta=1$ and covering radius $\rho=(m-1)/2$.\end{example}
\begin{example}\label{allex} Let $m=pq$ for some positive integer $p$, $Q=\{a_1,\ldots,a_q\}$,
  and $\alpha=(a_1^p,a_2^p,\ldots,a_q^p)$.  We define \[\All(pq,q)=\alpha^{L},\]
  the orbit of $\alpha$ under the top group of $\Aut(\Gamma)$.  The code $\All(pq,q)$ has minimum distance $\delta=2$ and covering
  radius $\rho=p(q-1)$. 
\end{example}
In \cite{diagnt} the authors proved that, along with the repetition code $\Rep(m,q)$, each of the above examples 
is neighbour transitive with $\Aut(C)=\Diag_m(S_q)\rtimes L$,
hence also diagonally neighbour transitive.  With $X=\Aut(C)$ it follows that $X^M=S_m$, $K=X\cap B=\Diag_m(S_q)$ and 
$X_1^Q=S_q$ in each case. Thus the code in Example \ref{weightex} is alphabet affine neighbour transitive, where as if $q\geq 5$, 
the codes in the remaining examples are alphabet almost simple neighbour transitive.  However, 
$\Rep(m,q)$ with $m\geq 3$ is the only code among these examples that has minimum distance $\delta\geq 3$.  We observe in this
case that $\soc(K)=\Diag_m(A_q)$ is transitive on $\Rep(m,q)$.  (Recall that $\soc(G)$, the socle of a group $G$, is the group generated
by the minimal normal subgroups of $G$.)

In the same paper the authors proved that any diagonally neighbour transitive code $C$ in $H(m,q)$  is either one 
of the codes in Examples \ref{repex}, \ref{injex}, \ref{weightex}, or $m=pq$ for some positive integer $p$ and $C$ is contained in $\All(pq,q)$, that is,
$C$ is a frequency permutation array.  Recall from Section \ref{powerline} that frequency permutation 
arrays have been studied recently due to possible applications to powerline communication,
with particular interest in permutation codes, the case where $p=1$.  Such codes give rise to further examples of neighbour 
transitive codes. 

\begin{example}\label{perm1}
To describe permutation codes we identify the alphabet $Q$ with the
set $\{1,\ldots,q\}$ and consider codes in the Hamming graph $\Gamma=H(q,q)$.  For $g\in S_q$ 
we define the vertex \[\alpha(g)=(1^g,\ldots,q^g)\in V(\Gamma).\]  
For $y\in S_q$ we let $x_y=(y,\ldots,y)\in B\cong S_q^q$, and we let $\sigma(y)$ denote the
automorphism induced by $y$ in $L\cong S_q$. It is known (see \cite[Lem. 8]{diagnt}) that for all $g,y\in S_q$, 
\[\alpha(g)^{x_y}=\alpha(gy)\textnormal{ and }\alpha(g)^{\sigma(y)}=\alpha(y^{-1}g).\]
Now, for a subset 
$T\subseteq S_q$, we define the \emph{permutation code generated by $T$} to be
the code \begin{equation}\label{ct}C(T)=\{\alpha(g)\in V(\Gamma)\,:\,g\in T\}.\end{equation} 
In \cite[Lem.~9 and Rem.~2]{diagnt} the authors showed that $C(T)$ is a neighbour transitive code with $\delta=2$ 
if and only if $T=S_q$.  Thus, if $T\neq S_q$, it holds that any neighbour transitive $C(T)$ has minimum distance $\delta\geq 3$.  
In the same paper the authors also proved that if $T$ is a subgroup of $S_q$ then $C(T)$ is diagonally
neighbour transitive if and only if $N_{S_q}(T)$ is $2$-transitive \cite[Thm.~2]{diagnt}. 
In particular, let \[A(T)=\{a_t=x_t\sigma(t)\,:\,t\in N_{S_q}(T)\}\]
and $X=\langle \Diag_q(T),A(T)\rangle$. It is shown in \cite[proof of Thm.~2]{diagnt} that, if $N_{S_q}(T)$ is $2$-transitive, 
then $C(T)$ is $X$-neighbour transitive.

Suppose now that $N_{S_q}(T)$ is $2$-transitive.  Then the intersection $K$ of the group $X$ with the base group $B$ is $K=\Diag_q(T)$, and so $\soc(K)=\Diag_q(\soc(T))$. It is clear
that $K$ is a normal subgroup of $X_1$, so $T\cong K^Q\unlhd X_1^Q$, and $X_1^Q\leq N_{S_q}(T)$.
Thus, if $T\neq S_q$ and $N_{S_q}(T)$ is $2$-transitive of almost simple type, then $C(T)$ and $X$ satisfy the conditions of Theorem \ref{main}.

In this case, if $T$ is simple then $\soc(K)$ acts transitively on $C(T)$.  
If $T$ is not simple then $\Delta=\alpha(1)^{\soc(K)}=C(\soc(T))$ is properly contained in $C(T)$.  Let $\mathcal{T}$ be a
transversal for $\soc(T)$ in $T$.  Then it follows that 
\[C(T)=\dot\bigcup_{t\in\mathcal{T}}\Delta^{x_t}=\dot\bigcup_{t\in\mathcal{T}} C(\soc(T)t).\]  
Note that if $T=S_q$, $C(T)$ is alphabet almost simple neighbour transitive and $C(T)$ is the disjoint union of two 
$\soc(K)$ orbits, but because $\delta=2$ the conditions of Theorem \ref{main} are not met.    
\end{example}

\begin{example} As in Example \ref{perm1}, identify $Q$ with the set $\{1,\ldots,q\}$.  The following
lemma proves that for any regular subgroup $T$ of $S_q$, the code $C(T)$ from \eqref{ct} is equivalent to $\Rep(q,q)$, and hence
neighbour transitive by Lemma \ref{dpartntreq}.

\begin{lemma}\label{regiffrep} Let $T$ be a non-trivial subgroup of $S_q$, and let $C(T)$ be as \eqref{ct} with minimum distance $\delta$. 
Then (i) $\delta=q$ if and only if $T$ acts semi-regularly on $Q$; and
(ii) $C(T)$ is equivalent to $\Rep(q,q)$ if and only if $T$ acts regularly on $Q$.
\end{lemma}

\begin{proof} (i) Assume that $\delta=q$. Suppose
  there exists $i\in Q$ such that the stabiliser $T_i\neq 1$.  Then
  there exists $1\neq g\in T_i$ such that $\alpha(1)|_i=\alpha(g)|_i=i$. In particular, 
  $d(\alpha(1),\alpha(g))\leq q-1$, contradicting the fact that
  $\delta=q$.  Therefore $T$ acts semi-regularly on $Q$.  Now assume
  that $T_i=1$ for all $i\in Q$.  Consider the distinct vertices
  $\alpha(g_1),\alpha(g_2)\in C(T)$ such that
  $d(\alpha(g_1),\alpha(g_2))=\delta$. If $\delta<q$ then $\alpha(g_1)|_i=\alpha(g_2)|_i$ for some 
  $i\in Q$, and so $i^{g_1}=i^{g_2}$.  Thus $i^{g_1g_2^{-1}}=i$,
  contradicting the fact that $T_i=1$ for all $i\in Q$.   

 (ii) Suppose that $C(T)$ is equivalent $\Rep(q,q)$. Then $C(T)$ has minimum distance $\delta=q$, so by (i), 
 $T_i=1$ for all $i\in Q$. Furthermore $|T|=|C(T)|=|\Rep(q,q)|=|Q|=q$.  Thus, the orbit stabiliser theorem
  implies that $T$ acts transitively on $Q$.  Hence $T$ acts regularly on $Q$.    
  Conversely suppose that $T$ acts regularly on $Q$.  Then by (i), $C(T)$ has minimum distance $\delta=q$.  Hence by Lemma \ref{cardc=q},
 $C(T)$ is equivalent to a code $C'\subseteq\Rep(q,q)$.  However, as $T$ acts regularly $Q$, it follows that $q=|T|=|C(T)|=|C'|$, so $C'=\Rep(q,q)$.
\end{proof}

\end{example}

\begin{example}\label{fgex} Let $G$ be a finite group of order $q$ and $\mathfrak{o}=(g_1,\ldots,g_q)$ be a fixed ordering of the elements of $G$.
Recall from Section \ref{fgsec} the code $C_{\mathfrak{o}}(G)$ in $\Gamma=H(q,q)$, the permutation code of $G$ with respect to $\mathfrak{o}$. 
In the definition of $C_{\mathfrak{o}}(G)$, we identify the identity element of $G$ with the vertex in $H(q,q)$ associated with the ordering $\mathfrak{o}$,
namely we define $\alpha_{\mathfrak{o}}(1)=(g_1,\ldots,g_q)\in H(q,q)$.
For another ordering $\mathfrak{o}'=(g_1',\ldots,g_q')$ of the elements of $G$, there exists $\sigma\in L\cong S_G$ (the top group of $\Aut(H(q,q))$) such that 
$\alpha_{\mathfrak{o}}(1)^\sigma=(g_{1\sigma^{-1}},\ldots,g_{q\sigma^{-1}})=(g_1',\ldots,g_q')=\alpha_{\mathfrak{o}'}(1)$ (see \eqref{autgpaction} for the
action of $L$ on vertices).
For any $g\in G$, we have 
\[\alpha_{\mathfrak{o}}(g)^\sigma=(g_1g,\ldots,g_qg)^\sigma=(g_1'g,\ldots,g_q'g)=\alpha_{\mathfrak{o}'}(g),\]
and so $C_{\mathfrak{o}}(G)^\sigma=C_{\mathfrak{o}'}(G)$. Thus, permutation codes of $G$ with
respect to different orderings of $G$ are equivalent, and we may therefore talk of $C(G)$ as
\emph{the} permutation code of $G$. We can now prove Theorem \ref{fgntthm}, which
states that $C(G)$ is $(S_G\times S_G)$-neighbour transitive.

\medskip\noindent\emph{Proof of Theorem~$\ref{fgntthm}$.}\quad Up to equivalence, the permutation code of $G$ is the code $C(r(G))$, where $r(G)\leq S_G$ is
the right regular representation of $G$.  As $r(G)$ acts regularly on $G$, it follows from Lemma \ref{regiffrep}--(ii) that $C(G)$ is equivalent
to $\Rep(q,q)$.  In the discussion following Example \ref{allex}, we saw that $\Rep(q,q)$ is $(\Diag_q(S_G)\rtimes L)$-neighbour transitive. Because
$\Diag_q(S_G)\rtimes L\cong S_G\times S_G$, the fact that $C(G)$ is $(S_G\times S_G)$-neighbour transitive follows from Lemma \ref{dpartntreq}.
\end{example}

\begin{table}[t]
\centering
\begin{tabular}{c c c c c}
\hline\hline
Line &$Degree=q$ & $T$ & Conditions & $\delta(C(T,T^\tau))$\\
\hline
1&6&$A_6\unlhd T\leq S_6 $&--& 8\\
2&11&$\PSL(2,11)$&--& 16\\
3&12&$M_{12}$&--& 16\\
4&15&$A_7$&--& 24\\
5&176&$HS$&--& 320\\
6&$(r^n-1)/(r-1)$&$\PSL(n,r)\unlhd T\leq \PGaL(n,r)$&$n>2$&$>\sqrt{q}-1$\\
\hline
\end{tabular}
\caption{$2$-transitive almost simple groups with two inequivalent actions.}\label{table1}
\end{table}

\begin{example}\label{perm2} Let $T$ be an almost simple $2$-transitive permutation group of degree $q$, with
socle $S$ (so $T\leq N_{S_q}(S)$), such that $N_{S_q}(S)$ is a proper subgroup of $\Aut(S)$.
By the classification of finite $2$-transitive groups, $T$, $q$ are as in one of the lines of Table \ref{table1} and $N_{S_q}(S)$ is a subgroup
of index $2$ in $\Aut(S)$ (see, for example, \cite[Table 7.4]{camperm}).
Moreover, there exists an outer automorphism $\tau$ of $T$ such that $\tau^2=1$.  For $t\in T$, we consider
the ordered pair \[\alpha(t,t^\tau)=(\alpha(t),\alpha(t^\tau))=(1^t,\ldots,q^t,1^{t^\tau},\ldots,q^{t^\tau}),\] and the code 
\[C(T,T^\tau)=\{\alpha(t,t^\tau)\,:\,t\in T\},\] which is a code in $\Gamma^2$ where $\Gamma=H(q,q)$.  
Let $(\alpha,\beta)$ be a general element of $\Gamma^2$ and let $\sigma$ be the automorphism in the top group 
that maps $(\alpha,\beta)$ to $(\beta,\alpha)$.  
Also let \[\Diag_m(T,T^\tau)=\{(x_t,x_{t^\tau})\in \Diag_m(T)^2\}\]
and \[A(T,T^\tau)=\{(a_t,a_{t^\tau})\in A(T)^2\},\] acting on $\Gamma^2$ as in \eqref{prodact}.
By letting $X=\langle \Diag_m(T,T^\tau), A(T,T^\tau), \sigma\rangle$, it follows that $C(T,T^\tau)$ is $X$-neighbour 
transitive \cite[Thm. 4.2]{twisted}.  

In this case, $K=\Diag_q(T,T^\tau)$ and $\soc(K)=\Diag_q(\soc(T),\soc(T)^\tau)$.
As in Example \ref{perm1}, we have that $T\cong K^Q\unlhd X_1^Q\leq N_{S_q}(T)$.
Moreover, $N_{S_q}(T)$ is $2$-transitive of almost simple type for each of the groups in Table \ref{table1}.
Therefore, if $C(T,T^\tau)$ has minimum distance at least $3$, then the code along with the group $X$
satisfy the conditions of Theorem \ref{main}.  Let us consider the minimum distance of $C(T,T^\tau)$.  

For $T$ as in line $1$ or $3$ of Table \ref{table1}, the minimum distance of $C(T,T^\tau)$
is given in \cite{twisted}. For $T$ as in lines $2$, $4$--$6$ of Table \ref{table1},
\[\delta(C(T,T^\tau))=2\times \delta(C(T)),\]
and the minimum distance of $C(T)$ is equal to the minimal degree of $T$ \cite[Sec.~4.4 \& 4.5]{twisted}.  
Consequently, the lower bound for the minimum distance given in line $6$ of Table \ref{table1} follows from the fact
that the minimal degree of a primitive permutation group of degree $q$ that does not contain $A_q$ is greater than 
$\frac{1}{2}(\sqrt{q}-1)$ \cite[Thm. 6.14]{babai}.  Therefore, it follows from Table \ref{table1} that $C(T,T^\tau)$ has minimum distance 
at least $3$, the only possible exception being the group $T$ in line $6$ with $(n,r)=(3,2)$.
However, it is straightforward to verify, using GAP \cite{gap}, that the minimal degree of $\PSL(3,2)$ is $4$, and so 
$\delta(C(T,T^\tau))=8$ in this case. 
This confirms that $C(T,T^\tau)$ and $X$ satisfy the conditions of Theorem \ref{main}.

For the groups $T$ in Table \ref{table1} that are simple, it holds that $\soc(K)$ acts transitively on $C(T,T^\tau)$.  If $T=S_6$
or $T$ is as in line $6$ of Table \ref{table1} with $\PSL(n,r)$ a proper subgroup of $T$, then $\Delta=\alpha(1,1)^{\soc(K)}$ is properly
contained in $C(T,T^\tau)$.  It then follows that \[C(T,T^\tau)=\dot\bigcup_{t\in\mathcal{T}}\Delta^{(x_t,x_{t\tau})}=
\dot\bigcup_{t\in\mathcal{T}} C(\soc(T)t,(\soc(T)t)^\tau)\] for a transversal $\mathcal{T}$ of $\soc(T)$ in $T$.
\end{example}

\section{The Structure of $\soc(K)$}\label{secstruct}
Let $C$ be an $X$-neighbour transitive code with $\delta\geq 3$ and $K:=X\cap B\neq 1$, and let 
\[\hat{Y}:=\varphi_1(X_1)=X_1^Q,\] where $\varphi_1$ is as in \eqref{varhom}.  
Recall from Proposition \ref{actionsprop} that $\hat{Y}$ acts $2$-transitively on the alphabet
$Q$, and is of almost simple type or of affine type \cite[Thm. 4.1B]{dixonmort}.  In particular, $\hat{Y}$ has a unique
minimal normal subgroup $T$ which is either non-abelian simple or elementary abelian.  
For the rest of this paper we assume that $T$ is non-abelian simple, so $\hat{Y}$ is of almost simple type.  
In this section we determine the structure of the socle of $K$, 
which we denote by $\soc(K)$. In particular, we prove, up to equivalence, 
that $\soc(K)$ is a sub-direct product of $T^m$, which then enables us to use Scott's Lemma \cite{scotts}
to determine the structure of $\soc(K)$ more explicitly.

First we introduce the following theorem, which was proved by the second
author with Schneider \cite[Theorem 1.1--(a)]{embeddingpap}.  The reader should note that we have
modified the notation of the original statement to suit our purposes.     

\begin{theorem}\label{embedding} Let $X\leq\Aut(\Gamma)=B\rtimes L$ and $\alpha\in H(m,q)$.
  Suppose that $X^M$ is transitive on $M$.  Then there exists
  $y\in B$ such  that $X^y\leq X_1^Q\wr X^M$.  Moreover, if $X_1^Q$
  is transitive on $Q$, $y$ may chosen so that $\alpha^y=\alpha$.    
\end{theorem}

\noindent It is a consequence of Proposition \ref{actionsprop} and Theorem \ref{embedding} that, for our $X$-neighbour transitive code $C$, 
there exists $y\in \Aut(\Gamma)$ such that $X^y\leq\hat{Y}\wr L$.  Therefore, by replacing $C$ with an equivalent code if necessary, Lemma \ref{dpartntreq} 
allows us to assume that $C$ is $X$-neighbour transitive with $X\leq\hat{Y}\wr L$ and $K:=X\cap \hat{Y}^m\neq 1$. 

Since $K$ is contained in the base group of $\Aut(\Gamma)$ it follows that $K\leq
X_i$ for each $i$.  We also note that \[K\leq \Pi_{i=1}^m K_i^Q\] where
$K_i^Q:=\varphi_i(K)\leq \hat{Y}$.  Now, because $K$ is a normal subgroup of $X$, and because $X$
acts transitively on $M$, we can prove that each $K_i^Q$ is conjugate to $K_1^Q$ in $\hat{Y}$.  
Thus there exists $x=(g_1,\ldots,g_m)\in\hat{Y}^m$ where each $g_i$ conjugates $K_i^Q$ to $K_1^Q$.  
By replacing $C$ by $C^x$ if necessary, Lemma \ref{dpartntreq} allows us to assume without loss of
generality that \[K\leq Y^m,\] where $Y=K_1^Q$.  If $Y=1$ then $K=1$, contradicting our assumption.  Therefore, because $K$ is a normal 
subgroup of $X_1$, it follows that $Y$ is a non-trivial normal subgroup of $\hat{Y}$, and so, as $\hat{Y}$ is almost simple, $T$ is also
the unique minimal normal subgroup of $Y$.  

\begin{proposition}\label{subdir} Let $G=Y^m$ so that $K\leq G$ as above, with $T$ the unique minimal normal subgroup of $Y$.  
Then $\soc(K)$ is a normal subgroup of $X$, and a sub-direct subgroup of $\soc(G)=T^m$.  
\end{proposition}

\begin{proof} We first prove that $\soc(G)=T^m$.  Let $H=T^m$, and for each $i\in M$ let $T_i=\{(1,\ldots,1,t,1,\ldots,1)\in
  G\,:\,t\in T\}$ where the non-trivial elements appear in the
  $i^{th}$ entry.  Thus we have $T_i\leq H$ and $H=\langle T_i\,:\,i\in M\rangle$.
  Since $T$ is a minimal normal subgroup of $Y$ it follows that, for 
  each $i$, $T_i$ is a minimal normal subgroup of $G$.  Thus
  $H\leq\soc(G)$.  Now let $R$ be any minimal normal subgroup of $G$.
  Since $H$ and $R$ are both normal in $G$ it follows that $R\cap H$
  is normal in $G$, and since $R$ is a minimal normal subgroup of $G$
  then either $R\cap H=1$ or $R\cap H=R$.  

  Suppose that $R\cap H=1$. Then, by the second isomorphism theorem, $R\cong RH/H\leq G/H$. 
  Schreier's Conjecture (which is known to be true by the classification of finite simple groups) states
 that $\Aut(T)/\Inn(T)\cong\Aut(T)/T$ is soluble.  Thus it follows that $Y/T$ is soluble, and therefore
 $(Y/T)^m\cong Y^m/T^m=G/H$ is soluble.  Consequently $R$ is soluble.  Since $R\neq 1$ there
  exists $i$ such that $1\neq \varphi_i(R)\unlhd Y$.  However, since $R$ is
  soluble it follows that $\varphi_i(R)$ is soluble which is a
  contradiction as $Y$ is almost simple.  Thus $R\leq H$ and so $H=\soc(G)$.  We also note 
  that following a similar argument proves that any non-abelian simple subgroup of $G$ is necessarily a subgroup of $H=\soc(G)$.  

  Now let $R$ be a minimal normal subgroup of $K$. We claim that $R\cong T$.  
  To prove this claim, we first consider the homomorphism $\varphi_i:K\longrightarrow S_q$
  defined in (\ref{varhom}).  As both $\ker(\varphi_i)$ and $R$ are normal in $K$, and because $R$ is minimal, 
  it follows that $R\cap\ker(\varphi_i)=1$ or $R$.  If $R\cap\ker(\varphi_i)=R$ for all $i$ then $\varphi_i(R)=1$
  for all $i$, which implies that $R=1$, a contradiction.  So there exists $j$ such that $R\cap\ker(\varphi_j)=1$, 
  which implies that $\varphi_j(R)\cong R$.  Moreover,
  $1\neq\varphi_j(R)\unlhd\varphi_j(K)=Y$ and so $T\leq \varphi_j(R)$
  since $Y$ is almost simple.  Thus $\soc(\varphi_j(R))=T$ and so
  $\soc(R)\cong T$.  Since $R$ is normal in $K$, and because $\soc(R)$ is 
  characteristic in $R$, it follows that $\soc(R)$ is normal in $K$.
  Therefore, as $R$ is minimal, we deduce that $R=\soc(R)$.  Thus
  $R\cong T$, and so $R$ is non-abelian simple.  Thus, from above, it holds that $R\leq\soc(G)$.  
  As this holds for every minimal normal subgroup of $K$, we deduce that $\soc(K)\leq\soc(G)$.  

  Now, we have shown that there exists $j$ such
  that $\varphi_j(R)\cong T$.  Thus, because $R\leq\soc(K)\leq\soc(G)=T^m$, we 
  conclude that $\varphi_j(\soc(K))\cong T$.  As $\soc(K)$ is a characteristic subgroup of $K$, and $K$ is a normal subgroup of 
 $X$, it follows that $\soc(K)$ is a normal subgroup of $X$. By 
  Proposition \ref{actionsprop}, $X$ acts transitively on $M$, so, by letting $X$ act on $\soc(K)$ via
  conjugation, it follows that $\varphi_i(\soc(K))\cong T$ for all $i$.  
\end{proof}

As $\soc(K)$ is a subdirect subgroup of a direct product of non-abelian simple groups, we can apply 
Scott's Lemma \cite[p.328]{scotts}.  In particular, Scott's Lemma implies that there exists a partition 
$\mathcal{J}=\{J_1,\ldots,J_\ell\}$ of $M$ such
that \[\soc(K)=D_1\times\ldots\times D_\ell\] where each $D_i$ is  
a full diagonal subgroup of $\Pi_{j\in J_i}T$.  We call $J_i$ the \emph{support of
  $D_i$}.  Let $x=(t_1,\ldots,t_m)\in D_i$.  Then $t_i=1$ for all
$i\in M\backslash\{J_i\}$, so we introduce the following notation.       

\begin{notation}\label{nota1} Let $J=\{i_1,\ldots,i_k\}\subseteq M$ and
  $t_{i_1},\ldots,t_{i_k}$ be $k$ permutations of $S_q$.
  Then we let
  $[t_{i_1},\ldots,t_{i_k}]_J$ denote the group element in the base
  group of $\Aut(\Gamma)$ given by 
 \[[t_{i_1},\ldots,t_{i_k}]_J|_u=\left\{\begin{array}{ll}
  t_u&\textnormal{if $u\in J$}\\
  1&\textnormal{if $u\notin J$} \end{array}\right.\]
\end{notation}

Now, because $D_i$ is a full diagonal subgroup of $\Pi_{j\in J_i}T$, it
follows that \[D_i=\{[t,t^{\psi_{i_2}},\ldots 
  t^{\psi_{i_k}}]_{J_i}\,:\,t\in T\}\] where each
$\psi_{i_s}$ is an automorphism of $T$.  Thus $D_i\cong T$ for each $i$, and so
$\soc(K)$ is the direct product of a finite set of non-abelian simple groups.  Hence,
$\mathcal{D}=\{D_1,\ldots,D_\ell\}$ is the complete set of minimal normal subgroups of $\soc(K)$ \cite[p.113]{dixonmort}.
By Proposition \ref{subdir}, $\soc(K)$ is normal in $X$, so $X$ acts on $\mathcal{D}$ by conjugation.  
In particular, for each $D_i\in\mathcal{D}$ and $x\in X$, there exists $D_u\in\mathcal{D}$ such that $x^{-1}D_ix=D_u$.  

\begin{lemma}\label{ninva} Let $x=h\sigma\in X$.  Then $x^{-1}D_ix=D_j$ if and only
  if $J_i^\sigma=J_j$.  
\end{lemma}

\begin{proof} Let $1\neq y\in D_i$, and suppose that $J_i^\sigma=J_j$. As $y_s\neq 1$ for all $s\in D_i$, it
follows that $x^{-1}yx|_s\neq 1$ for all $s\in J_j$.  Hence $x^{-1}yx\in D_j$, and so $x^{-1}D_ix=D_j$ by the comments preceding this lemma.
Conversely suppose that $x^{-1}D_ix=D_j$, and let $s\in J_i$, so $y_s\neq1$.  Then $x^{-1}yx|_{s^\sigma}\neq 1$, and because $x^{-1}yx\in D_j$ 
it follows that $s^\sigma\in J_j$.
\end{proof}

\begin{remark}\label{remninva} As $X$ acts on $\mathcal{D}$ by conjugation, it is
a consequence of Lemma \ref{ninva} that $\mathcal{J}$ is an $X$-invariant partition of $M$.
Moreover, because $X$ acts transitively on $M$ it also follows that $X$ acts transitively on $\mathcal{D}$.
\end{remark}

Let us consider the group \[D_i=\{[t,t^{\psi_{i_2}},\ldots
  t^{\psi_{i_k}}]_{J_i}\,:\,t\in T)\}\] and the group
 $\overline{N_{S_q}(T)}\leq \Aut(T)$.  (Here $\overline{N_{S_q}(T)}$ denotes
that subgroup of $\Aut(T)$ induced by $N_{S_q}(T)$.)  Let $\mathcal{T}$ be a
 transversal for $\overline{N_{S_q}(T)}$ in $\Aut(T)$.  Then for each
 $s\in J_i$, $\psi_{i_s}=z_{i_s}\bar{h}_{i_s}$ for some $z_{i_s}\in\mathcal{T}$ and
 $\bar{h}_{i_s}\in \overline{N_{S_q}(T)}$.  Conjugating $D_i$ by
 $[1,h_{i_2}^{-1},\ldots,h_{i_k}^{-1}]_{J_i}$
 yields \[D_i^{[1,h_{i_2}^{-1},\ldots,h_{i_k}^{-1}]_{J_i}}=\{[t,t^{z_{i_2}},\ldots,t^{z_{i_{k}}}]_{J_i}\,:\,t\in   
 T\}.\] Let $y=\Pi_{i=1}^\ell[1,h^{-1}_{i_2},\ldots,h^{-1}_{i_k}]_{J_i}$.
 By replacing $C$ by $C^y$ if necessary, we may assume without loss of
 generality that for all $D_i$, each $\psi_{i_s}$ lies in the
 transversal $\mathcal{T}$.   

\begin{definition}\label{defform} Let $\mathcal{T}$ be a transversal for
  $\overline{N_{S_q}(T)}$ in $\Aut(T)$ and 
  $D_i=\{[t,t^{\psi_{i_2}},\ldots t^{\psi_{i_k}}]_{J_i}\,:\,t\in T)\}$
  where each $\psi_{i_s}\in\mathcal{T}$.  If $\psi_{i_s}=1$ for all
  $i_s\in J_i$ we say that \emph{$D_i$ has Form 1}, otherwise we say that \emph{$D_i$ has Form 2}.     
\end{definition} 

Using the classification of finite almost simple 2-transitive groups we know that $|\mathcal{T}|\leq 2$ \cite[Table 7.4]{camperm}.  
If $|\mathcal{T}|=1$ then $D_i$
has to have $\Form 1$.  If $D_i$ has $\Form
2$ then $T$ is the socle of one of the groups from the third column of Table
\ref{table1} and $\mathcal{T}=\{1,\tau\}$ for some $\tau\in
\Aut(T)\backslash\overline{N_{S_q}(T)}$ \cite[Table 7.4]{camperm}.  Moreover we can assume 
that $\tau^2=1$ (see \cite[Remark 4.1]{twisted}). Recall that, by assumption, each $\psi_{i_s}$ is equal to $1$ or $\tau$.
 
\begin{lemma}\label{form} Every minimal normal subgroup of $\soc(K)$ has the same Form,
i.e. all the $D_i$ have $\Form 1$, or they all have $\Form 2$.     
\end{lemma}

\begin{proof}  Suppose there exist $D_i, D_j\in\mathcal{D}$ such that  $D_i$ has $\Form 1$ and $D_j$ has $\Form 2$.
  Let $J_i=\{i_1,\ldots,i_k\}$, $J_j=\{j_1,\ldots,j_k\}$ be the respective supports of $D_i, D_j$.  As $D_j$ has $\Form 2$, there
 exists $s\in J_j$ such that $\varphi_s=\tau$. By Remark \ref{remninva},
 there exists $x=(h_1,\ldots, h_m)\sigma\in X$ such that $x^{-1}D_ix=D_j$, and so $J_i^{\sigma}=J_j$.  
 Let $u,v\in J_i$ such that $u^\sigma=j_1$ and $v^\sigma=s$. Then for each
 $t\in T$ there exists $t'\in T$ such that $[t,\ldots,t]_{J_i}^x=[t',t'^{\psi_{j_2}},\ldots,t'^{\psi_{j_k}}]_{J_j}$.
  This implies that for each $t\in T$ there exists $t'\in T$ such that
  $t^{h_u}=t'$ and $t^{h_s}=t'^\tau$, which in turn implies that
  $t^{h_s}=t^{h_u\tau}$.  Hence, as automorphisms of
  $T$, we have that $\bar{h}_s=\bar{h}_{u}\tau$.  However this
  implies that $\tau\in\overline{N_{S_q}(T)}$ which is a
  contradiction. 
\end{proof}

Suppose there exists $D_i\in\mathcal{D}$, with support $J_i$, that has $\Form 2$.  
We define \begin{equation}\label{ji12}J_i^{(1)}=\{s\in J_i\,:\,\psi_s=1\}\textnormal{ and } 
 J_i^{(2)}=\{s\in J_i\,:\,\psi_s=\tau\}.\end{equation}  It follows that $J_i$ is
 the disjoint union of $J_i^{(1)}$ and $J_i^{(2)}$.  
  Because $\mathcal{J}$ is a partition of $M$ and since $S_m$ acts
 $m$-transitively on $M$, there exists
 $\sigma\in\Pi\Sym(J_i)\leq L\cong S_m$ such
 that \[D_i^\sigma=\{[t,\ldots,t,t^{\tau},\ldots,t^{\tau}]_{J_i}\,:\,t\in     
 T\}\] for each $i$.  Therefore, replacing $C$ by $C^\sigma$, we can
 assume that each $D_i$ has this form.  We want to be able to refer to
 the two possibilities for $D_i$.  Therefore we define the following.

\begin{definition}\label{case1case2}  We say $\case1$ holds if $\Form
  1$ holds for all $D_i$ (that is, $\psi_s=1$ for all $s\in M$) 
  and \[D_i=\{[t,\ldots,t]_{J_i}\,:\,t\in T\}\] for
  $1\leq i\leq \ell$.  In this case we abbreviate $[t,\ldots,t]_{J_i}$
  (as in Notation \ref{nota1}) by $[t]_{J_i}$  where \[ [t]_{J_i}|_u=\left\{\begin{array}{ll}
  t&\textnormal{if $u\in J_i$}\\
  1&\textnormal{if $u\notin J_i$.} \end{array}\right.\]  We say
  $\taucase$ holds if $T$ is the socle of one of the groups from the third column of Table
  \ref{table1} and $q$ is the corresponding degree;
  $\tau\in\Aut(T)\backslash\overline{N_{S_q}(T)}$ such that
  $\tau^2=1$; $\Form 2$ holds for all $D_i$ (that is there exists 
  $s\in J_i$ such that $\psi_s\neq 1)$; and \[D_i=\{[t,\ldots,t,t^{\tau},\ldots,t^{\tau}]_{J_i}\,:\,t\in 
  T\}\] for $1\leq i\leq \ell$.  In this case we abbreviate
  $[t,\ldots,t,t^{\tau},\ldots,t^{\tau}]_{J_i}$ (as in Notation
  \ref{nota1}) by $[t,t^\tau]$ where \[ [t,t^\tau]_{J_i}|_u=\left\{\begin{array}{ll}
  t&\textnormal{if $u\in J_i^{(1)}$}\\
  t^\tau&\textnormal{if $u\in J_i^{(2)}$}\\
  1&\textnormal{if $u\notin J_i$.} \end{array}\right.\]
\end{definition}

\begin{lemma}\label{ninvaonj} Suppose $D_i\in\mathcal{D}$, with support $J_i$,  has $\Form 2$.  Then
$\{J_i^{(1)}, J_i^{(2)}\}$ is an $X_{J_i}$-invariant partition of $J_i$, and in particular $|J_i^{(1)}|=|J_i^{(2)}|$.
\end{lemma}

\begin{proof} Let $x=h\sigma=(h_1,\ldots,h_m)\sigma\in X_{J_i}$ and
 suppose  $\emptyset\neq (J_i^{(1)})^\sigma \cap J_i^{(1)}\subset J_i^{(1)}$, that is, properly contained in $J_i^{(1)}$.  
 Then there exist $u,s\in J_i^{(1)}$ such that $u^\sigma\in J_i^{(1)}$ and $s^\sigma\in J_i^{(2)}$.  
 Following a similar argument to that used in the proof of Lemma \ref{form}, we deduce that, as automorphisms of $T$, $\bar{h}_s=\bar{h}_u\tau$, 
 This implies that $\tau\in \overline{N_{S_q}(T)}$, which is a contradiction.  Thus either $(J_i^{(1)})^\sigma=J_i^{(1)}$ or 
$J_i^{(1)}\cap J_i^{(1)}=\emptyset$.  A similar argument shows this to be true for $J_i^{(2)}$ also.  The result now follows from
the fact that $X_{J_i}$ acts transitively on $J_i$.  
\end{proof}

\begin{remark}\label{remkeven} (i) It is a consequence of Lemma \ref{ninvaonj} that 
$\mathcal{J}^{(\tau)}=\{J_1^{(1)},J_1^{(2)},
\ldots,J_\ell^{(1)},J_\ell^{(2)}\}$ is an $X$-invariant partition
of $M$.       

 (ii) Let $|\mathcal{D}|=\ell$.  As $X$ acts transitively on $M$,
  $|J_i|=m/\ell=k$ for all $i$, in particular, $m=\ell k$.  If $\taucase$
  holds, Lemma \ref{ninvaonj} implies that $k$ is even.  Let
  $\alpha=(\alpha_1,\ldots,\alpha_m)\in C$, $J\in\mathcal{J}$ and $i\in J$ if $\case1$ holds, and $i\in J^{(1)}$ if $\taucase$ holds.
  Since $T$ acts transitively on $Q$, there exists $t\in T$ such that $\alpha_i^t\neq\alpha_i$. 
  Let $x=[t]_J$ in $\case 1$ and $x=[t,t^\tau]_J$ in $\taucase$, so $x\in\soc(K)$.  As $\alpha_i\neq\alpha_i^t$, it follows that
  $\alpha\neq\alpha^x$, so $3\leq\delta\leq d(\alpha,\alpha^x)$.  Hence $k\geq 3$, and if $\taucase$ holds, $k\geq 4$ as $k$ is even.   
\end{remark}

\section{The structure of the Projection codes}\label{secstructproj}

In the previous section we proved that there exists an $X$-invariant partition
$\mathcal{J}$ of $M$ for any alphabet almost simple $X$-neighbour transitive
code with $\delta\geq 3$ and $K:=X\cap B\neq 1$.  Moreover, if $\taucase$ holds (as in Definition \ref{case1case2})
Lemma \ref{ninvaonj} implies that $\mathcal{J}^{(\tau)}$ is also an $X$-invariant partition of $M$.
Let $J\in\mathcal{J}$ if $\case1$ holds and $J\in\mathcal{J}^{(\tau)}$ if
$\taucase$ holds.  Then Lemma \ref{projcode} implies that $\pi_J(C)$ is a
$\chi(X_J)$-neighbour transitive code in $H(J,q)$.  In this section we
prove Proposition \ref{applyproj}, which describes the code $\pi_J(C)$
in each case.  Recall from Example \ref{repex} the repetition code
$\Rep(k,q)$ in $H(k,q)$, and from Example \ref{allex} the code
$\All(pq,q)$ in $H(pq,q)$. 

\begin{proposition}\label{applyproj} Let $C$ be an alphabet almost simple $X$-neighbour transitive code with
  $\delta\geq 3$ and $K\neq 1$.  Then if $\case1$ holds, either
\begin{itemize}
\item[(i)] $\pi_J(C)=\Rep(k,q)$ in $H(J,q)$ with $|J|=k$ for all $J\in\mathcal{J}$, or
\item[(ii)] $\pi_J(C)\subseteq \All(pq,q)$ in $H(J,q)$ with $|J|=pq$ for some positive integer $p$ for all $J\in\mathcal{J}$.
\end{itemize}
If $\taucase$ holds then $\pi_J(C)\subseteq \All(pq,q)$ in $H(J,q)$ with $|J|=pq$ for some positive integer $p$ for all $J\in\mathcal{J}^{(\tau)}$. 
\end{proposition}

\begin{proof}
Let $J\in\mathcal{J}$ if $\case1$ holds and $J\in\mathcal{J}^{(\tau)}$ if $\taucase$ holds.  Also let $|J|=r$ and denote $H(J,q)$ by $\Gamma(J)$.  
As $\soc(K)$ is contained in the base group of $\Aut(\Gamma)$, it follows that $\soc(K)\unlhd\,X_J$.
Hence $\chi(\soc(K))\unlhd\chi(X_J)$ in $\Aut(\Gamma(J))$.  Thus $\chi(X_J)\leq N:=N_{\Aut(\Gamma(J))}(\chi(\soc(K)))$.
We claim that \begin{equation}\label{normal} N=\Diag_r(N_{S_q}(T))\rtimes S_r.\end{equation}
By Remark \ref{chirem}, $\chi(\soc(K))=\Diag_r(T)$, and it is clear that the top group $L(J)\cong S_r$
of $\Aut(\Gamma(J))\cong S_q\wr S_r$ centralises $\chi(\soc(K))$, so $L(J)\leq N$.  Hence
if $h\sigma\in N$ it follows that $h=(h_{i_1},\ldots,h_{i_r})\in N$.  Now, for $t\in T$ it follows that
\[(t,\ldots,t)^h=(t^{h_{i_1}},\ldots,t^{h_{i_r}})\in\chi(\soc(K)).\]
Thus, for all $t\in T$ and $j\geq 2$, $t^{h_{i_1}}=t^{h_{i_j}}$.  This
implies that $h_{i_1}h^{-1}_{i_j}\in C_{S_q}(T)$.  However, because $T$ is almost simple and
acts primitively on $Q$, $C_{S_q}(T)=1$ \cite[Theorem 4.2A]{dixonmort}.  As $h\in N$ it follows
that for each $t\in T$ there exists $t'\in T$ such that $(t,\ldots,t)^h=(t',\ldots,t')$.  In particular, this implies that $h_{i_1}\in N_{S_q}(T)$. 
Thus $N\leq \Diag_r(N_{S_q}(T))\rtimes S_r$, and it is straight forward to show that these two groups are in fact equal.  Hence the claim holds.

Now, because $J$ is a block of imprimitivity for the action of $X$ on $M$, Lemma \ref{projcode} implies that $\pi_J(C)$ is a
$\chi(X_J)$-neighbour transitive code in $\Gamma(J)$.  As $\chi(X_J)\leq \Diag_r(N_{S_q}(T))\rtimes S_r$, 
we conclude that $\pi_J(C)$ is a diagonally neighbour transitive code in $\Gamma(J)$.  Hence,
by the classification of diagonally neighbour transitive codes \cite{diagnt}, one of the following holds:
\begin{itemize}
\item[(a)] $\pi_{J}(C)$ is the repetition code $\Rep(r,q)$
\item[(b)] $r<q$ and
  $\pi_{J}(C)=\{(\alpha_{i_1}\,\ldots,\alpha_{i_r})\,:\,\alpha_u\neq\alpha_v$
  for all $u,v\in J\}$, 
\item[(c)] $r$ is odd, $q=2$ and $\pi_{J}(C)=\{\alpha\in \Gamma(J)\,:\,\alpha$ has
  weight $\frac{r-1}{2}$ or $\frac{r+1}{2}\}$, or
\item[(d)] there exists an integer $p$ such that $r=pq$ and
  $\pi_{J}(C)\subseteq\All(pq,q)$ in $\Gamma(J)$. 
\end{itemize}
Let $\delta_J$ be the minimum distance of $\pi_{J}(C)$.  In each of the cases 
(a)--(d), $\pi_{J}(C)$ is not the complete code in $\Gamma(J)$.
Therefore Corollary \ref{projcodedist} implies that $\delta_J\geq 2$.
If (b) holds then we saw in Example \ref{injex} that $\delta_J=1$,
which is a contradiction.  Suppose that (c) holds.
Then $q=2$ and $T\leq S_2$.  However, $S_2$ is not almost simple, which
is a contradiction. Therefore either (a) or (d) holds for $\pi_J(C)$. 
We note that if $\pi_J(C)=\Rep(r,q)$ then $\delta_J=r$.   We claim
that if $\pi_J(C)\subseteq\All(pq,q)$ then $2\leq\delta_J<r$.   

When (d) holds, the parameter $r$ is equal to $pq$ for some positive integer $p$.
Also, by Remark \ref{chirem}, $\chi(\soc(K))=\Diag_r(T)$.  Since $\pi_J(C)$ is contained in
$\All(pq,q)$, which has minimum distance $2$, it follows
that $\delta_J\geq 2$.  Now suppose that $\delta_J=r$.  Then
as $\pi_J(C)$ is neighbour transitive, Lemma \ref{cardc=q} implies that 
$\pi_J(C)$ is equivalent to $\Rep(r,q)$.  In particular $|\pi_J(C)|=q$.  Let $\alpha\in\pi_J(C)$ and 
suppose there exist $t_1,t_2\in T$ such that \[\alpha^{(t_1,\ldots,t_1)}=\alpha^{(t_2,\ldots,t_2)}.\]  Then,
because every letter of $Q$ appears in $\alpha$, it follows that $a^{t_1}=a^{t_2}$ for all $a\in Q$.  That is
$t_1=t_2$.  Hence $\chi(\soc(K))$ acts regularly on its orbits in $\pi_J(C)$.  Therefore $|T|\leq|\pi_J(C)|=q$.
In particular, as $T$ acts transitively on $Q$, this implies that $|T|=q$ and that $T$ acts regularly on $Q$.
Thus, by \cite[Theorem 4.2A]{dixonmort}, $C_{S_q}(T)$ acts transitively on $Q$, contradicting the fact that $C_{S_q}(T)=1$.  
Hence $\delta_J<r$ and the claim is proved. Therefore, to recap, we have shown that $\delta_J\geq 2$, and either (a) holds, or (d) holds 
with $\delta_J<r$. 

Now let $J_i,J_j\in\mathcal{J}$ and consider the codes $\pi_{J_i}(C)$ and $\pi_{J_j}(C)$ with minimum distances $\delta_{J_i}$ and $\delta_{J_j}$.
By Lemma \ref{projneigh2}, it follows that $\delta_{J_i}=\delta_{J_j}$.  Thus, because the code in (a) has minimum distance $\delta_J=r$ and the code
in (d) has minimum distance $2\leq\delta_J<r$, we conclude that if $\case1$ holds then either (i) or (ii) in the statement holds.

Assume now that $\taucase$ holds, and recall from Definition \ref{case1case2} that $T$ is the socle of one of the 
groups from the third column of Table~\ref{table1}.
Consider $J\in\mathcal{J}$, so $J^{(1)}$,
$J^{(2)}\in\mathcal{J}^{(\tau)}$.  If we project onto $J$, we saw in the previous section that
\[\chi(\soc(K))=\{(\underbrace{t,\ldots,t}_{k/2},\underbrace{t^\tau,\ldots,t^\tau}_{k/2})\,|\,t\in T)\}.\]  
By Lemma \ref{ninvaonj}, $\{J^{(1)},J^{(2)}\}$ is a $\chi(X_J)$-invariant partition of $J$.
Moreover, by \eqref{normal} we have that
$\chi(X_{J^{(1)}})$ and $\chi(X_{J^{(2)}})$ are subgroups of $\Diag_{k/2}(N_{S_q}(T))\rtimes S_{k/2}.$
Hence, for each $x\in\chi(X_{J})$ there exist $h_1,h_2\in N_{S_q}(T)$ and $\sigma\in S_{k/2}\wr S_2$ such that
\[x=(\underbrace{h_1,\ldots,h_1}_{k/2},\underbrace{h_2,\ldots,h_2}_{k/2})\sigma.\]  Thus for each $t\in T$, it follows that 
\[\begin{array}{ccc}
  x^{-1}(t,\ldots,t,t^\tau,\ldots,t^\tau)x&=&
\begin{cases} (t^{\overline{h}_1},\ldots,t^{\overline{h}_1},t^{\tau \overline{h}_2},\ldots,t^{\tau \overline{h}_2}) &\text{if $\sigma$ stabilises $J^{(1)}$}\\ 
(t^{\tau \overline{h}_2},\ldots,t^{\tau \overline{h}_2},t^{\overline{h}_1},\ldots,t^{\overline{h}_1}) &\text{otherwise.}
\end{cases}
\end{array}\]
As $\chi(\soc(K))\unlhd\chi(X_{J})$, we deduce in both cases that $t^{\overline{h}_2}=t^{\tau\overline{h}_1\tau}$ for all
$t\in T$ (recall that $\tau$ was chosen so $\tau^2=1$).  
Now, because $\overline{N_{S_q}(T)}\cong N_{S_q}(T)$ is a normal subgroup of $\Aut(T)$ for each of the possible groups $T$, one can deduce 
that $\overline{h}_2=\tau\overline{h}_1\tau=\overline{h_1^\tau}$.  This implies that $h_2h_1^{-\tau}\in C_{S_q}(T)=1$, so $h_2=h_1^{\tau}$.
Thus for each $x\in\chi(X_{J})$ 
there exist $h\in N_{S_q}(T)$ and $\sigma\in S_{k/2}\wr S_2$ such that \[x=(h,\ldots,h,h^{\tau},\ldots,h^{\tau})\sigma.\]
Now suppose that either 
$\pi_{J^{(1)}}(C)$ or $\pi_{J^{(2)}}(C)$ is the repetition code.  By considering their minimum distances and applying a similar argument 
to the one above, we deduce that both codes are the repetition code in their respective Hamming graphs.  
Since $\mathcal{J}$ is an $X$-invariant partition of $M$, Proposition
\ref{projcode} implies that $\pi_{J}(C)$ is also a
$\chi(X_{J})$-neighbour transitive code.  As both 
$\pi_{J^{(1)}}(C)$ and $\pi_{J^{(2)}}(C)$ are the repetition codes, we deduce that every codeword
of $\pi_{J}(C)$ is of the
form \[(\underbrace{a,\ldots,a}_{k/2},\underbrace{b,\ldots,b}_{k/2})\]  
for some $a,b\in Q$.  Thus $\delta_{J}=k/2$ or $k$.
Let $\alpha=(a,\ldots,a,b,\ldots,b)\in\pi_{J}(C)$. For each possible group $T$, 
$(N_{S_q}(T)_a)^\tau$ has two orbits on $Q$,
each of length at least $2$.  (Here $(N_{S_q}(T)_a)^\tau$ denotes the stabiliser of $a$ in $N_{S_q}(T)$ under the automorphism $\tau$.)
Let $c,d\in Q\backslash\{b\}$ such that $b,c$ are in the same
$(N_{S_q}(T)_a)^\tau$-orbit and $b,d$ are not in the same 
$(N_{S_q}(T)_a)^\tau$-orbit.  Now let
$\nu_1=\nu(\alpha,{k/2+1},c)$ and
$\nu_2=\nu(\alpha,{k/2+1},d)$, which are both neighbours of
$\alpha$.  Since $\pi_{J}(C)$ is $\chi(X_{J})$-neighbour
transitive, there exists $x=(h,\ldots,h,h^\tau,\ldots,h^\tau)\sigma\in
\chi(X_{J})$ such that $\nu_1^x=\nu_2$. Suppose
$(J^{(1)})^\sigma=J^{(2)}$.  Then, as $\nu_1^x=\nu_2$, it follows
that $a^h=b$ and $a^h=d$, which is a contradiction given that $b\neq
d$.  Therefore $(J^{(1)})^\sigma=J^{(1)}$, and so $a^h=a$ and
$h^\tau\in (N_{S_q}(T)_a)^\tau$.  However, it then follows that either $c^{h^\tau}=d$ or $b^{h^\tau}=d$, contradicting the choice
of $b,c$ and $d$.   Hence we conclude that neither $\pi_{J^{(1)}}(C)$ nor $\pi_{J^{(2)}}(C)$ is the repetition code. 
Thus (d) holds for both $\pi_{J^{(1)}}(C)$ and $\pi_{J^{(2)}}(C)$.  If follows from this argument that $\pi_{J^*}(C)$ cannot be the repetition code for
any $J^*\in\mathcal{J}^{(\tau)}$, that is, $\pi_{J^*}(C)\subseteq\All(pq,q)$ for all $J^*\in\mathcal{J}^{(\tau)}$.  
\end{proof}

\section{Building Blocks of $C$}\label{secbuildingblocks}

Let $C$ be an alphabet almost simple $X$-neighbour transitive code with $\delta\geq 3$ and $K:=X\cap B\neq 1$,
and recall from Section \ref{secstruct} that there exists an $X$-invariant partition $\mathcal{J}$ of $M$.  Let $\hat{C}$ denote the 
projection code $\pi_J(C)$ for some $J\in\mathcal{J}$, and let $k=|J|$.  
Also let $\hat{S}=\chi(\soc(K))$.  In this section we describe certain $\hat{S}$-orbits in $\hat{C}$.  We then use these to describe a $\soc(K)$-orbit in $C$.  

\subsection{}\label{secdescribe} Assume that $\case1$ holds, so 
\[\hat{S}=\{x_t=(t,\ldots,t)\,:\,t\in T\}.\]  Let $\alpha\in\hat{C}$
and define \[\hat{\Delta}=\alpha^{\hat{S}},\] the $\hat{S}$-orbit containing $\alpha$.  
By Proposition \ref{applyproj}, either $\hat{C}$ is the repetition code or $\hat{C}\subseteq\All(pq,q)$ where $p=k/q$ is a positive integer. 
Suppose that the former holds.  Then there exists $a\in Q$ such that $\alpha=(a,\ldots,a)$.   
For $t\in T$ it follows that \[\alpha^{x_t}=(a,\ldots,a)^{x_t}=(a^t,\ldots,a^t).\]
Because $T$ is acting transitively on $Q$, we deduce that $\hat{\Delta}=\hat{C}$.    

Now suppose that $\hat{C}\subseteq\All(pq,q)$ with $p=k/q$, and let us identify $Q$ with the set $\{1,\ldots, q\}$.  
As every letter of $Q$ appears $p$ times in $\alpha$, there exists $\sigma$ in the top group of $\Aut(H(k,q))$ such that 
\[\alpha^\sigma=(1,2,\ldots,q,\ldots,1,\ldots,q)=(\alpha(1),\ldots,\alpha(1))=\rep_p(\alpha(1)).\]  By replacing $\hat{C}$ with $\hat{C}^\sigma$ if
necessary, we can assume that $\alpha=\rep_p(\alpha(1))\in\hat{C}$.  Note that as $\hat{S}$ is centralised by the top group of
$\Aut(H(k,q))$, $\hat{S}$ is left unchanged by doing this.  Now let $x_t\in\hat{S}$.  It follows that
\begin{equation}\label{calc1}\rep_p(\alpha(1))^{x_t}=(1,\ldots,q,\ldots,1,\ldots,q)^{x_t}=
(1^t,\ldots,q^t,\ldots,1^t,\ldots,q^t)=\rep_p(\alpha(t)).\end{equation}
As $\hat{S}=\{x_t\,:\,t\in T\}$, we deduce that \[\hat{\Delta}=\alpha^{\hat{S}}=\Rep_p(C(T)).\]  

\subsection{}\label{secdescribe2} Suppose that $\taucase$ holds and let $\alpha\in\hat{C}$.  We saw in Remark \ref{remkeven} that 
$k$ is even, and by Proposition \ref{applyproj}, $\pi_{J^{(i)}}(C)\subseteq\All(pq,q)$ with $p=k/2q$ for $i=1,2$.  
Thus every letter in $Q$ appears in $\alpha$ exactly $p$ times on
$J^{(1)}$, and similarly for $J^{(2)}$.  Consequently there exists $\sigma$ in the top group of $\Aut(H(k,q))$ that fixes $J^{(1)}$ setwise such that 
\begin{equation}\label{alphaform}\alpha^\sigma=(\alpha(1),\ldots,\alpha(1))\end{equation} which consists of $2p$ repeated copies of $\alpha(1)$. 
As before, by replacing $\hat{C}$ if necessary, we can assume that $\alpha$ is as in \eqref{alphaform}.  Now, in this case 
\[\hat{S}=\{(\underbrace{t,\ldots,t}_{pq},\underbrace{t^\tau,\ldots,t^\tau}_{pq})\,:\,t\in T\}.\]
In order to give a nice description of the $\hat{S}$-orbit of $\alpha$, we again conjugate $\hat{C}$ by an element in 
the top group so that $\hat{S}$ looks slightly different.  Let \[\sigma'=\Pi_{i=1}^p\Pi_{j=1}^q((2i-1)q+j\,,\,(2i-1)pq+j).\]
and replace $\hat{C}$ by $\hat{C}^{\sigma'}$.  Note that $\alpha^{\sigma'}=\alpha$ and
\begin{equation}\label{sform}\hat{S}^{\sigma'}=\{x^{(\tau)}_t=(\underbrace{t,\ldots,t}_q,\underbrace{t^\tau,\ldots,t^\tau}_q,\ldots,
\underbrace{t,\ldots,t}_q,\underbrace{t^\tau,\ldots,t^\tau}_q)\,:\,t\in T\},\end{equation} so that after replacement we can assume that $\alpha$ 
is as in \eqref{alphaform} and that $\hat{S}$ has the form above.  We can identify $\alpha$ with the vertex $\rep_p(\alpha(1,1^\tau))$ where 
$\alpha(t,t^\tau)=(\alpha(t),\alpha(t^\tau))$ for any $t\in T$.  Once we've made this identification, it follows from a direct calculation similar to \eqref{calc1}
that for $t\in T$, \[\alpha^{x^{(\tau)}_t}=\rep_p(\alpha(1,1^\tau))^{x^{(\tau)}_t}=\rep_p(\alpha(t,t^{\tau})).\]
Hence we deduce that \[\hat{\Delta}=\alpha^{\hat{S}}=\Rep_p(C(T,T^\tau)).\] 

\subsection{Piecing the parts back together}\label{secpiece}

Let $\alpha\in C$ and \[\Delta=\alpha^{\soc(K)}.\] As $\soc(K)$ is equal to the direct product of the groups $D_i\in\mathcal{D}$, it follows
that we can identify $\Delta$ with the Cartesian product of the $D_i$-orbits on $\alpha$.  That is
\[\Delta=\alpha^{D_1}\times\ldots\times\alpha^{D_\ell}\]
Because each $D_i$ has support $J_i$, it follows that $D_i$ leaves $\alpha$ unchanged on the set of entries $M\backslash{J_i}$.  
So we can identify $\alpha^{D_i}$ with $\pi_{J_i}(\alpha^{\chi(D_i)})=\pi_{J_i}(\alpha)^{\chi(D_i)}$.  
(The idea here is that we are throwing away the part of $\alpha$ that is left unchanged when $D_i$ acts on it.)
We also note that $\chi(D_i)=\chi(\soc(K))$ when we project onto $J_i$.  
Hence, in each case, by replacing $C$ with an equivalent code if necessary, we can identify $\Delta$ with the Cartesian product of the orbit $\hat{\Delta}$ 
described in Section \ref{secdescribe} or Section \ref{secdescribe2}.

More specifically, suppose that $\case1$ holds with $\hat{\Delta}=\pi_J(C)=\Rep(k,q)$ for all $J\in\mathcal{J}$.  Then $\Delta$ is equal to 
the Cartesian product of $\ell$-copies of the repetition code.  This is just the product construction applied to $\Rep(k,q)$, as defined in 
Section \ref{secprodrep}, so 
\[\Delta=\Prod_{\ell}(\Rep(k,q)).\] Let $\beta\in C$.  As $\pi_J(\beta)\in\Rep(k,q)$ 
for all $J\in\mathcal{J}$, there exist $a_1,\ldots,a_\ell\in Q$ such
that $\beta=(a_1^k,\ldots,a_\ell^k)$.  In particular, $\beta\in\Delta$, so in this case 
\begin{equation}\label{des1}C=\Delta=\Prod_{\ell}(\Rep(k,q)).\end{equation}

Suppose now that $\case1$ holds such that for all $J\in\mathcal{J}$, $\pi_J(C)\subseteq\All(pq,q)$ and $k=pq$ for some positive integer $p$.  
Then there exists $\sigma\in\Pi_{i=1}^{\ell}\Sym(J_i)$ that centralises $\soc(K)$ such that $\alpha^\sigma=(\alpha(1),\ldots,\alpha(1))$. 
Hence, by replacing $C$ by $C^{\sigma}$ if necessary, it follows that $\Delta$ is the Cartesian product of $\hat{\Delta}=\Rep_p(C(T))$.  
This is just the product construction applied to $\Rep_p(C(T))$, that is, \begin{equation}\label{des2}\Delta=\Prod_{\ell}(\Rep_p(C(T))).\end{equation}  
If $\taucase$ holds then we choose $\sigma\in\Pi_{i=1}^{\ell}\Sym(J_i)$  so that $\alpha^\sigma=(\alpha(1),\ldots,\alpha(1))$ but also so that each
$\chi(D_i)$ is as in \eqref{sform}.  By replacing $C$ with $C^\sigma$, it follows 
that \begin{equation}\label{des3}\Delta=\Prod_\ell(\Rep_p(C(T,T^\tau))).\end{equation}

\section{Proof of Theorem \ref{main}}\label{secproofofthm}

Let $C$ be an alphabet almost simple $X$-neighbour transitive code with $\delta\geq 3$ and $K:=X\cap B\neq 1$ in $H(m,q)$.  
Then $X_1^Q$ is a $2$-transitive group of almost simple type.  Let
$T$ be the minimal normal subgroup of $X_1^Q$, which is non-abelian simple.  By replacing $C$ by an equivalent
code if necessary, it follows from Proposition \ref{subdir} that $\soc(K)$ is a subdirect product of $T^m$.
Thus, applying Scott's Lemma \cite{scotts}, there exists a partition 
$\mathcal{J}=\{J_1,\ldots,J_\ell\}$ of $M$ such
that \[\soc(K)=D_1\times\ldots\times D_\ell\] where each $D_i$ is  
a full diagonal subgroup of $\Pi_{j\in J_i}T$, and by Lemma \ref{ninva}, $\mathcal{J}$ is an $X$-invariant partition of $M$.  
By again replacing $C$ if necessary, it follows from Lemma \ref{form} that all the subgroups $D_i$ have the same Form, as in Definition \ref{defform}, and that either
$\case1$ or $\taucase$ holds, as in Definition \ref{case1case2}.  Moreover, if $\taucase$ holds, we deduce from Lemma \ref{ninvaonj} that there exists an
$X$-invariant partition $\mathcal{J}^{(\tau)}$ of $M$ which is a refinement of the partition $\mathcal{J}$.  

Now let $\alpha\in C$, $\Delta=\alpha^{\soc(K)}$ and $k=m/\ell$.  We deduce from Section \ref{secpiece} that, up to equivalence, either 
\begin{equation*}\Delta=\left\{\begin{array}{ll}     
\Prod_\ell(\Rep(k,q))&\textnormal{if $\case1$ and Proposition \ref{applyproj}--(i) hold,}\\
\Prod_\ell(\Rep_p(C(T)))\textnormal{ with $k=pq$}&\textnormal{if $\case1$ and Proposition \ref{applyproj}--(ii) hold, or}\\
\Prod_\ell(\Rep_p(C(T,T^\tau)))\textnormal{ with $k=2pq$} &\textnormal{if $\taucase$ holds.}
 \end{array}\right.\end{equation*}
As $\Delta$ is a $\soc(K)$-orbit, and because $\soc(K)$ is a normal subgroup of $X$, $\Delta$ is a block of imprimitivity for the action of $X$ on $C$.  
Thus either $C=\Delta$ (which is necessarily true if $\case1$ and Proposition \ref{applyproj}--(i) hold), or
$C$ is the disjoint union of $X$-translates of $\Delta$.  Moreover, Lemma \ref{codeimprim} implies that $\Delta$ is neighbour transitive, which 
we claim is true in each case. In Example \ref{repex}, we saw that $\Rep(k,q)$ is neighbour transitive, and in \cite{diagnt}, the authors proved that 
for a subgroup $H\leq S_q$, the code $C(H)$ is neighbour transitive if and only if $N_{S_q}(H)$ is $2$-transitive.  
Moreover, in joint work with Spiga \cite{twisted}, the authors proved that for each $T$ in Table \ref{table1}
and outer automorphism $\tau$ with order $2$, the code $C(T,T^\tau)$ is neighbour transitive.  
Thus it follows from Lemma \ref{prodconstnt} and \cite[Lemma 5]{diagnt} that in each case, $\Delta$ is indeed a neighbour transitive code, proving the claim.  
Finally, we observe that if $\case1$ and Proposition \ref{applyproj}--(ii) hold or $\taucase$ holds, 
then $\Delta$ is a frequency permutation array with each letter appearing 
$\ell p$ or $2\ell p$ times respectively.

\section{Another Example}\label{secexas}

In this section we demonstrate that for some of the codes $C$ in Theorem \ref{main}, the projected codes $\pi_J(C)$ may have
minimum distance smaller than that of $C$, and indeed, may have minimum distance $2$.
We give an example of an alphabet almost simple $X$-neighbour
transitive code with $\delta=3$ and $X$-invariant partition $\mathcal{J}$ of $M$
such that $\pi_J(C)$ has minimum distance $\delta(\pi_J(C))=2$ for each $J\in\mathcal{J}$.  

\begin{example} Let $Q=\{1,\ldots,q\}$ for some $q\geq 5$ and define 
\[C=\{(\alpha(t_1),\ldots,\alpha(t_\ell))\in H(\ell q, q)\,:\,t_it_j^{-1}\in A_q\,\,\forall i,j\}.\] 
Let $R=\Diag_q(S_q)\rtimes S_q\leq\Aut(H(q,q))$ and
\[X=\{(x_{h_1}\sigma_1,\ldots,x_{h_\ell}\sigma_\ell)\sigma\in R\wr
  S_\ell \,:\,h_ih_j^{-1}\in A_q, \sigma_i\sigma_j^{-1}\in
  A_q\textnormal{ for all $i,j$}\}.\] We claim that $C$ is $X$-neighbour transitive.  

  \begin{proof}  Let $\beta=(\alpha(t_1),\ldots,\alpha(t_\ell))\in C$ and
  $x=(x_{h_1}\sigma_1,\ldots,x_{h_\ell}\sigma_\ell)\sigma\in X$. 
  It follows from \eqref{prodact} and \cite[Lemma 8]{diagnt} that
\begin{eqnarray*}\beta^x&=&(\alpha(t_1)^{x_{h_1}\sigma_1},\ldots,\alpha(t_\ell)^{x_{h_\ell}\sigma_\ell})^\sigma\\ 
&=&(\alpha(\sigma_1^{-1}t_1h_1),\ldots,\alpha(\sigma_\ell^{-1}t_\ell h_\ell))^\sigma\\
&=&(\alpha(\sigma_{1{\sigma^{-1}}}^{-1}t_{1{\sigma^{-1}}}h_{1{\sigma^{-1}}}),\ldots
,\alpha(\sigma_{\ell{\sigma^{-1}}}^{-1}t_{\ell{\sigma^{-1}}}h_{\ell{\sigma^{-1}}})).\end{eqnarray*}
  From the definition of $C$ and $X$, we deduce that for all $i,j$, 
  \[\sigma_i^{-1}t_ih_i(\sigma_j^{-1}t_jh_j)^{-1}=\sigma_i^{-1}t_ih_ih_j^{-1}t_i^{-1}t_it_j^{-1}\sigma_i\sigma_i^{-1}\sigma_j\in
  A_q.\] Therefore $X\leq\Aut(C)$.  Now let
  \[\alpha=(\alpha(1),\ldots,\alpha(1))\in C.\] Then
  $y=(x_{t_1},\ldots,x_{t_\ell})\in X$ and $\alpha^y=\beta$.  Since
  $\beta$ was arbitrarily chosen, it follows that $X$ acts
  transitively on $C$.  

  To prove that $X$ acts transitively on the neighbour set of $C$, we first describe the neighbours of $\alpha$.  
  The neighbours of $\alpha(1)$ in $H(q,q)$ are
  \[\Gamma_1(\alpha(1))=\{\nu(\alpha(1),a,b)\,:\,\textnormal{$a,b\in
    Q$ and $a\neq b$}\}.\]  Thus following the notation of Section \ref{secprodrep}, the neighbours of $\alpha$ in $H(\ell q,q)$ are
  \[\Gamma_1(\alpha)=\{\gamma(\alpha,i,\nu(\alpha(1),a,b))\,:\,i\in
  \{1,\ldots,\ell\},\textnormal{ $a,b\in Q$ and $a\neq b$}\}.\]
  Consider the group \[W=\{(x_{h_1}\sigma_1,\ldots,x_{h_\ell}\sigma_\ell)\sigma\in 
  X\,:\,h_i=\sigma_i\,\,\forall i\}.\]  Then $W\leq X_\alpha$.  Let
  \[\nu_1=\gamma(\alpha,i,\nu(\alpha(1),a,b))\textnormal{ , }\nu_2=\gamma(\alpha,j,\nu(\alpha(1),s,u))\] be two neighbours of
  $\alpha$.  Since $A_q$ acts $2$-transitively on $Q$, there exists
  $h\in A_q$ such that $a^h=s$ and $b^h=u$.  Let
  $x=(x_h\sigma,\ldots,x_h\sigma)\in W$.  Since $\sigma=h$, 
  we deduce from Lemma \ref{neiact2} and \eqref{prodact} that
  $\nu_1^x=\gamma(\alpha,i,\nu(\alpha(1),s,u))$.  It follows that there
  exists $\sigma'\in S_\ell$ such that $i^{\sigma'}=j$ and so
  $\nu_1^{x\sigma'}=\nu_2$.  Thus $X_\alpha$ acts transitively on
  $\Gamma_1(\alpha)$, and so, because $X$ acts transitively on $C$, we deduce that $X$ acts transitively
  on the neighbour set of $C$.  Hence $C$ is $X$-neighbour transitive.       
 \end{proof}

  Since $C\subseteq\Prod_\ell (C(S_q))$ and $\delta(C(S_q))=2$, it follows
  that $C$ has minimum distance $\delta\geq 2$.  If $\delta=2$, then because
  $X$ acts transitively on $C$, there exists $\beta\in C$ such that $d(\alpha,\beta)=2$.  However, this holds 
  if and only if $\beta=(\alpha(1),\ldots,\alpha(t_i),\ldots,\alpha(1))$ for some
  transposition $t_i\in S_q$, and such a vertex is not a codeword.  Thus $\delta\geq 3$.  Let $t_1$ be a
  $3$-cycle in $A_q$ and
  $\beta=(\alpha(t_1),\alpha(1),\ldots,\alpha(1))\in C$.  Then 
  $d(\alpha,\beta)=3$.    Hence $C$ has minimum distance $\delta=3$.
  Now, it is clear that $\mathcal{J}=\{J_1,\ldots,J_\ell\}$, where $J_i=\{a+(i-1)q\,:\,a\in 
  Q\}$, is an $X$-invariant partition of $M$.  Because
  $(\alpha(t),\ldots,\alpha(t))\in C$ for all $t\in S_q$, we have that
  $\pi_J(C)=C(S_q)$ for all $J\in\mathcal{J}$.  Thus 
  $C$ is an example of an $X$-neighbour transitive code with minimum distance $3$ and $X$-invariant partition $\mathcal{J}$ such that
  $\pi_J(C)$ has minimum distance $2$ for each $J\in\mathcal{J}.$

  We observe that $K=X\cap B=\{(x_{h_1},\ldots,x_{h_\ell})\,:\,h_ih_j^{-1}\in A_q\}$.
  Also, $(x_h,\ldots,x_h)\in   
  K$ for all $h\in S_q$ and so $X_1^Q=S_q$.  Thus $C$ and $X$ satisfy the conditions of Theorem \ref{main}.  Now
  let $G=\Pi_{i=1}^\ell \Diag_q(S_q)$.  By following arguments that are similar to those used in the proof of Proposition \ref{subdir}, we 
  deduce that $\soc(K)$ is a subgroup of $\soc(G)=\Pi_{i=1}^\ell \Diag_q(A_q)$.  
  Moreover, for each $i$, $T_i=\{(1,\ldots,x_t,\ldots,1)\,:\,x_t\in
  \Diag_q(A_q)\}$ is a minimal normal subgroup of $K$,  so $T_i\leq\soc(K)$. 
  Consequently $\soc(K)=\soc(G)$.  It follows
  that \[\Delta=\alpha^{\soc(K)}=\{(\alpha(h_1),\ldots,\alpha(h_\ell))\,:\,h_i\in
  A_q\}=\Prod_\ell(C(A_q)),\] which is a proper subset of $C$.  Now, if $\beta=(\alpha(t_1),\ldots,\alpha(t_\ell))\in C$, then $t_i\in A_q$ for some $i$
  if and only if $t_j\in A_q$ for all $j$.  From this we deduce that
 \[C=\dot\bigcup_{t\in\mathcal{T}}\Prod_\ell(C(A_qt))=\dot\bigcup_{t\in\mathcal{T}}\Delta^{(x_t,\ldots,x_t)}\] where $\mathcal{T}$ is 
 a transversal for $A_q$ in $S_q$.
\end{example}

\section{Acknowledgements}

This research for the first author was supported by an Australian Postgraduate Award and by 
a grant from the University of Western Australia associated with 
the Australian Research Council Federation Fellowship FF0776186 of the second author.


\begin{thebibliography}{10}
\providecommand{\url}[1]{{#1}}
\providecommand{\urlprefix}{URL }
\expandafter\ifx\csname urlstyle\endcsname\relax
  \providecommand{\doi}[1]{DOI~\discretionary{}{}{}#1}\else
  \providecommand{\doi}{DOI~\discretionary{}{}{}\begingroup
  \urlstyle{rm}\Url}\fi

\bibitem{babai}
Babai, L.: On the order of uniprimitive permutation groups.
\newblock Ann. of Math. (2) \textbf{113}(3), 553--568 (1981)

\bibitem{bailey}
Bailey, R.F.: Error-correcting codes from permutation groups.
\newblock Discrete Math. \textbf{309}(13), 4253--4265 (2009)

\bibitem{best}
Best, M.R.: Perfect codes hardly exist.
\newblock IEEE Trans. Inform. Theory \textbf{29}(3), 349--351 (1983)

\bibitem{blake}
Blake, I.: Permutation codes for discrete channels.
\newblock Information Theory, IEEE Transactions on \textbf{20}(1), 138 -- 140
  (1974)

\bibitem{blakeetal}
Blake, I.F., Cohen, G., Deza, M.: Coding with permutations.
\newblock Inform. and Control \textbf{43}(1), 1--19 (1979)

\bibitem{conjecex}
Borges, J., Rif{\`a}, J., Zinoviev, V.A.: On non-antipodal binary completely
  regular codes.
\newblock Discrete Math. \textbf{308}(16), 3508--3525 (2008)

\bibitem{rho=2}
Borges, J., Rif{\`a}, J., Zinoviev, V.A.: On {$q$}-ary linear completely
  regular codes with {$\rho=2$} and antipodal dual.
\newblock Adv. Math. Commun. \textbf{4}(4), 567--578 (2010)

\bibitem{distreg}
Brouwer, A.E., Cohen, A.M., Neumaier, A.: Distance-regular graphs,
  \emph{Ergebnisse der Mathematik und ihrer Grenzgebiete (3) [Results in
  Mathematics and Related Areas (3)]}, vol.~18.
\newblock Springer-Verlag, Berlin (1989)

\bibitem{cald2}
Calderbank, A.: Nonexistence of a uniformly packed [70,58,5] code (corresp.).
\newblock Information Theory, IEEE Transactions on \textbf{32}(6), 828--833
  (1986)

\bibitem{cald1}
Calderbank, A.R., Goethals, J.M.: On a pair of dual subschemes of the {H}amming
  scheme {$H_n(q)$}.
\newblock European J. Combin. \textbf{6}(2), 133--147 (1985)

\bibitem{camperm}
Cameron, P.J.: Permutation groups, \emph{London Mathematical Society Student
  Texts}, vol.~45.
\newblock Cambridge University Press, Cambridge (1999)

\bibitem{chu1}
Chu, W., Colbourn, C.J., Dukes, P.: Constructions for permutation codes in
  powerline communications.
\newblock Des. Codes Cryptogr. \textbf{32}(1-3), 51--64 (2004)

\bibitem{chu2}
Chu, W., Colbourn, C.J., Dukes, P.: On constant composition codes.
\newblock Discrete Appl. Math. \textbf{154}(6), 912--929 (2006)

\bibitem{chu3}
Colbourn, C., Klove, T., Ling, A.: Permutation arrays for powerline
  communication and mutually orthogonal latin squares.
\newblock Information Theory, IEEE Transactions on \textbf{50}(6), 1289--1291
  (2004)

\bibitem{delsarte}
Delsarte, P.: An algebraic approach to the association schemes of coding
  theory.
\newblock Philips Res. Rep. Suppl. (10), vi+97 (1973)

\bibitem{dixonmort}
Dixon, J.D., Mortimer, B.: Permutation groups, \emph{Graduate Texts in
  Mathematics}, vol. 163.
\newblock Springer-Verlag, New York (1996)

\bibitem{ferr}
Ferreira, H., Vinck, A.J.H.: Interference cancellation with permutation trellis
  codes.
\newblock In: Vehicular Technology Conference, 2000. IEEE-VTS Fall VTC 2000.
  52nd, vol.~5, pp. 2401--2407 vol.5 (2000)

\bibitem{frankl}
Frankl, P., Deza, M.: On the maximum number of permutations with given maximal
  or minimal distance.
\newblock J. Combinatorial Theory Ser. A \textbf{22}(3), 352--360 (1977)

\bibitem{gap}
The GAP~Group, http://www.gap-system.org: GAP -- Groups, Algorithms,
  and Programming, Version 4.7.4 (2014)

\bibitem{diagnt}
Gillespie, N., Praeger, C.: Diagonally neighbour transitive codes and frequency
  permutation arrays.
\newblock Journal of Algebraic Combinatorics \textbf{39}(3), 733--747 (2014)

\bibitem{ngthesis}
Gillespie, N.I.: {Neighbour Transitivity on Codes in {H}amming Graphs}.
\newblock Ph.D. thesis, The University of Western Australia (2011)

\bibitem{ntransintro}
Gillespie, N.I., Praeger, C.E.: Neighbour transitivity on codes in {H}amming
  graphs.
\newblock Des. Codes Cryptogr. \textbf{67}(3), 385--393 (2013)

\bibitem{twisted}
Gillespie, N.I., Praeger, C.E.: Twisted permutation codes (2014).
\newblock ArXiv:1402.5305

\bibitem{comtran}
Gillespie, N.I., Praeger, C.E., Giudici, M.: Classification of a family of
  completely transitive codes (2012).
\newblock ArXiv:1208.0393

\bibitem{giudici}
Giudici, M., Praeger, C.E.: Completely transitive codes in {H}amming {G}raphs.
\newblock European Journal of Combinatorics \textbf{20}(7), 647 -- 662 (1999)

\bibitem{hanvinck}
Han~Vinck, A.J.: Coded modulation for power line communications.
\newblock AE\"U Journal \textbf{Jan}, 45--49 (2000).
\newblock ArXiv:1104.4528v1

\bibitem{hill}
Hill, R.: A first course in coding theory.
\newblock Oxford Applied Mathematics and Computing Science Series. The
  Clarendon Press Oxford University Press, New York (1986)

\bibitem{huffpless}
Huffman, W.C., Pless, V.: Fundamentals of error-correcting codes.
\newblock Cambridge University Press, Cambridge (2003)

\bibitem{keevash}
Keevash, P., Ku, C.Y.: A random construction for permutation codes and the
  covering radius.
\newblock Des. Codes Cryptogr. \textbf{41}(1), 79--86 (2006)

\bibitem{lincostello}
Lin, S., Jr., D.J.C.: Error Control Coding: Fundamentals and Applications.
\newblock Prentice-Hall (1983)

\bibitem{lind1}
Lindstr{\"o}m, K.: The nonexistence of unknown nearly perfect binary codes.
\newblock Ann. Univ. Turku. Ser. A I (169), 28 (1975)

\bibitem{lind2}
Lindstr{\"o}m, K.: All nearly perfect codes are known.
\newblock Information and Control \textbf{35}(1), 40--47 (1977)

\bibitem{martin}
Martin, W.J.: Completely regular codes --- a viewpoint and some problems.
\newblock Proceedings of 2004 Com2MaC Workshop on Distance-Regular Graphs and
  Finite Geometry pp. 43--56 (2004)

\bibitem{neu}
Neumaier, A.: Completely regular codes.
\newblock Discrete Math. \textbf{106/107}, 353--360 (1992).
\newblock A collection of contributions in honour of Jack van Lint

\bibitem{stateoftheart}
Pavlidou, N., Han~Vinck, A., Yazdani, J., Honary, B.: Power line
  communications: state of the art and future trends.
\newblock Communications Magazine, IEEE \textbf{41}(4), 34 -- 40 (2003).
\newblock \doi{10.1109/MCOM.2003.1193972}

\bibitem{peterson}
Peterson, W.W., Weldon Jr., E.J.: Error-correcting codes, second edn.
\newblock The M.I.T. Press, Cambridge, Mass.-London (1972)

\bibitem{pless}
Pless, V.: Introduction to the theory of error-correcting codes, second edn.
\newblock Wiley-Interscience Series in Discrete Mathematics and Optimization.
  John Wiley \& Sons Inc., New York (1989).
\newblock A Wiley-Interscience Publication

\bibitem{embeddingpap}
Praeger, C.E., Schneider, C.: Embedding permutation groups into wreath products
  in product action.
\newblock J. Aust. Math. Soc. \textbf{92}(1), 127--136 (2012)

\bibitem{binctrarb}
Rif{\`a}, J., Zinoviev, V.A.: On a class of binary linear completely transitive
  codes with arbitrary covering radius.
\newblock Discrete Math. \textbf{309}(16), 5011--5016 (2009)

\bibitem{kronprod}
Rif{\`a}, J., Zinoviev, V.A.: New completely regular {$q$}-ary codes based on
  {K}ronecker products.
\newblock IEEE Trans. Inform. Theory \textbf{56}(1), 266--272 (2010)

\bibitem{liftperfect}
Rif{\`a}, J., Zinoviev, V.A.: On lifting perfect codes.
\newblock IEEE Trans. Inform. Theory \textbf{57}(9), 5918--5925 (2011)

\bibitem{roman}
Roman, S.: Coding and information theory, \emph{Graduate Texts in Mathematics},
  vol. 134.
\newblock Springer-Verlag, New York (1992)

\bibitem{scotts}
Scott, L.L.: Representations in characteristic {$p$}.
\newblock In: The {S}anta {C}ruz {C}onference on {F}inite {G}roups ({U}niv.
  {C}alifornia, {S}anta {C}ruz, {C}alif., 1979), \emph{Proc. Sympos. Pure
  Math.}, vol.~37, pp. 319--331. Amer. Math. Soc., Providence, R.I. (1980)

\bibitem{shannon}
Shannon, C.E.: A mathematical theory of communication.
\newblock Bell System Tech. J. \textbf{27}, 379--423, 623--656 (1948)

\bibitem{slepian}
Slepian, D.: Permutation modulation.
\newblock Proc. IEEE \textbf{53}(3), 228--236 (1965)

\bibitem{smith2}
Smith, D., Montemanni, R.: Permutation codes with specified packing radius.
\newblock Designs, Codes and Cryptography \textbf{69}(1), 95--106 (2013)

\bibitem{sole}
Sol{\'e}, P.: Completely regular codes and completely transitive codes.
\newblock Discrete Math. \textbf{81}(2), 193--201 (1990)

\bibitem{tiet}
Tiet{\"a}v{\"a}inen, A.: On the nonexistence of perfect codes over finite
  fields.
\newblock SIAM J. Appl. Math. \textbf{24}, 88--96 (1973)

\bibitem{vantil}
van Tilborg, H.C.A.: Uniformly packed codes.
\newblock Technische Hogeschool Eindhoven, Eindhoven (1976).
\newblock With a Dutch summary, Doctoral dissertation, University of Technology
  Eindhoven

\bibitem{extendconst}
Zinoviev, V.A., Rif{\`a}, J.: On new completely regular {$q$}-ary codes.
\newblock Problemy Peredachi Informatsii \textbf{43}(2), 34--51 (2007)

\end{thebibliography}

\end{document}